\documentclass{amsart}
\usepackage{hyperref}
\usepackage{amsmath,amsthm,amscd,amssymb}
\usepackage{enumerate}
\usepackage{curves}

\newcommand*{\mailto}[1]{\href{mailto:#1}{\nolinkurl{#1}}}

\newtheorem{theorem}{Theorem}[section]
\newtheorem{lemma}[theorem]{Lemma}

\newtheorem{remark}[theorem]{Remark}

\newtheorem{hypothesis}[theorem]{Hypothesis {\bf H.}\hspace*{-0.6ex}}

\numberwithin{equation}{section}

\newcommand{\C}{\mathbb{C}}
\newcommand{\R}{\mathbb{R}}
\newcommand{\N}{\mathbb{N}}
\newcommand{\E}{\mathrm{e}}
\newcommand{\I}{\mathrm{i}}

\newcommand{\sipmu}{\sigma_\pm^{\mathrm{u}}}
\newcommand{\sipml}{\sigma_\pm^{\mathrm{l}}}

\newcommand{\lau}{\lambda^{\mathrm{u}}}
\newcommand{\lal}{\lambda^{\mathrm{l}}}

\newcommand{\Res}{\mathop{\rm Res}}
\newcommand{\sign}{\mathop{\rm sign}}

\renewcommand{\Im}{\mathop{\rm Im}}
\renewcommand{\Re}{\mathop{\rm Re}}

\newcommand{\clos}{\mathop{\rm clos}}
\newcommand{\inte}{\mathop{\rm int}}
\newcommand{\beq}{\begin{equation}}
\newcommand{\eeq}{\end{equation}}
\newcommand{\bal}{\begin{align}}
\newcommand{\eal}{\end{align}}
\newcommand{\nn}{\nonumber}

\newcommand{\la}{\lambda}

\def\XXint#1#2#3{{\setbox0=\hbox{$#1{#2#3}{\int}$}
     \vcenter{\hbox{$#2#3$}}\kern-.5\wd0}}

\numberwithin{equation}{section}

\begin{document}

\title[Scattering theory for Schr\"odinger operators]{Scattering theory for Schr\"odinger operators on steplike, almost periodic infinite-gap backgrounds}

\author[K. Grunert]{Katrin Grunert}
\address{Faculty of Mathematics\\ Nordbergstrasse 15\\ 1090 Wien\\ Austria}
\email{\mailto{katrin.grunert@univie.ac.at}}
\urladdr{\url{http://www.mat.univie.ac.at/~grunert/}}

\thanks{Research supported by the Austrian Science Fund (FWF) under Grant No.\ Y330.}


\begin{abstract}
We develop direct scattering theory for one--dimensional Schr\"odinger operators with steplike potentials, which are asymptotically close to different Bohr almost periodic infinite--gap potentials on different half--axes.
\end{abstract}

\maketitle

\section{Introduction}

One of the main tools for solving various Cauchy problems, since the seminal work of Gardner, Green, Kruskal, and Miura \cite{GGKM} in 1967, is the inverse scattering transform and therefore, since then, a large number of articles has been devoted to direct and inverse scattering theory. 

Given two (in general different) one-dimensional background Schr\"odinger operators $L_\pm$ with real finite-gap potentials $p_\pm(x)$, i.e.
\begin{equation}
 L_\pm=-\frac{d^2}{dx^2}+p_\pm(x), \quad x\in\R,
\end{equation}
one can consider the perturbed one-dimensional Schr\"odinger operator
\begin{equation}
 L=-\frac{d^2}{dx^2}+p(x),\quad x\in\R,
\end{equation}
where $p(x)$ satisfies a second moment condition, i.e.
\begin{equation}
 \pm\int_0^{\pm \infty} (1+x^2)\vert p(x)-p_\pm(x)\vert dx<\infty.
\end{equation}
Then one of the main tools when considering the scattering problem for the Schr\"odinger operator $L$, are the transformation operators which map the background Weyl solutions of the operators $L_\pm$ to the Jost solutions of $L$. In particular, if the background operators are well-understood, the transformation operators enable us to perform the direct scattering step, which means to characterize the scattering data and to derive the Gel'fand-Levitan-Marchenko equation. The starting point for the inverse scattering step is the Gel'fand-Levitan-Marchenko equation together with the scattering data, from which one deduces the kernels of the transformation operators and recovers the potential $p(x)$.

In much detail the scattering problem has been studied in the case where $p(x)$ is asymptotically close to $p_\pm(x)=0$. For a complete investigation and discussions on the history of this problem we refer to the monographs of Levitan \cite{L} and Marchenko \cite{M}. Taking this as a starting point, two natural extension have been considered. On the one hand the case of steplike constant asymptotics $p_\pm(x)=c_\pm$, where $c_+\not =c_-$ denote some constants, has been investigated by Buslaev and Fomin \cite{BF}, Cohen and Kappeler \cite{CK}, and Davies and Simon \cite{DS}. On the other hand Firsova \cite{F1} studied the case of equal periodic, finite-gap potentials $p_+(x)=p_-(x)$. Rather recently, the combination of these two cases, namely the case that the initial condition is asymptotically close to steplike, quasi-periodic, finite-gap potentials $p_-(x)\not =p_+(x)$, has been investigated by Boutet de Monvel, Egorova, and Teschl \cite{BET}. Trace formulas in the case of one periodic background were given by Mikikits-Leitner and Teschl \cite{MT} and a Paley--Wiener theorem in Egorova and Teschl \cite{ET}.

Of course the inverse scattering theory is also the main ingredient for solving the Cauchy problem of the Korteweg--de Vries (KdV) equation via the inverse scattering transform \cite{M}. Moreover, scattering theory is also the basic ingredient for setting up the associated Riemann--Hilbert problem from which the long-time asymptotics can be derived via the nonlinear steepest descent analysis (see \cite{GT} for an overview). In the case of finite-gap backgrounds the Cauchy problem was solved
by Grunert, Egorova, and Teschl \cite{GET}, \cite{ET3}. Note that the analogous Cauchy problem for the modified KdV equations can be obtained via the Miura transform \cite{ET2}. The long-time asymptotics in case of one finite-gap background were recently derived by Mikikits-Leitner and Teschl \cite{MT2}.

Of much interest is also the case of asymptotically periodic solutions, which has been first considered by Firsova \cite{F1}. In the present work we propose a complete investigation of the direct scattering theory for Bohr almost periodic infinite-gap backgrounds, which belong to the so--called Levitan class. It should be noticed, that this class, as a special case, includes the set of smooth, periodic infinite--gap operators. 

To set the stage, we need:

\begin{hypothesis}\label{Hyp}
Let
\beq\nn
0\leq E_0^\pm< E_1^\pm<  \dots  <E_n^\pm< \dots
\eeq 
be two increasing  sequences of points on the real axis which satisfy the following conditions:
\begin{enumerate}
\item
for a certain $l^\pm>1$, $\sum_{n=1}^\infty (E_{2n-1}^\pm)^{l^\pm} (E_{2n}^\pm-E_{2n-1}^\pm) <\infty$ and 
\item 
$E_{2n+1}^\pm-E_{2n-1}^\pm > C^\pm n^{\alpha^\pm}$, where $C^\pm$ and $\alpha^\pm$ are some fixed, positive constants.
\end{enumerate}
\end{hypothesis}

We will call, in what follows, the intervals $(E_{2j-1}^\pm,E_{2j}^\pm)$ for $j=1,2,\dots$ gaps. In each closed gap $[E_{2j-1}^\pm,E_{2j}^\pm]$ , $j=1,2,\dots$, we choose a point $\mu_j^\pm$ and an arbitrary sign $\sigma_j^\pm\in\{-1,1\}$.

Next consider the system of differential equations for the functions $\mu_j^\pm(x)$, $\sigma_j^\pm(x)$, $j=1,2,...$, which is an infinite analogue of the  well-known Dubrovin equations, given by 
\begin{align}\label{repmujx}
\frac{d \mu_j^\pm(x)}{d x}=
& -2\sigma_j^\pm(x)\sqrt{-(\mu_j^\pm(x)-E_0^\pm)}\sqrt{\mu_j^\pm(x)-E_{2j-1}^\pm}\sqrt{\mu_j^\pm(x)-E_{2j}^\pm}  \\ \nn 
&\times \prod_{k=1, k\not =j}^\infty  \frac{\sqrt{\mu_j^\pm(x)-E_{2k-1}^\pm}\sqrt{\mu_j^\pm(x)-E_{2k}^\pm}}{\mu_j^\pm(x)-\mu_k^\pm(x)}
\end{align}
with initial conditions $\mu_j^\pm(0)=\mu_j^\pm$ and $\sigma_j^\pm(0)=\sigma_j^\pm$, $j=1,2,\dots$
\footnote{We will use the standard branch cut of the square root in the domain $\mathbb C\setminus\mathbb R_+$ with $\Im\sqrt z>0$.}. Levitan \cite{L}, \cite{L2}, and \cite{L3}, proved, that
 this system of differential equations is uniquely solvable, that the solutions $\mu_j^\pm(x)$, $j=1,2,\dots$ are continuously differentiable and satisfy $\mu_j^\pm(x)\in [E_{2j-1}^\pm, E_{2j}^\pm]$ for all $x\in\R$. Moreover, these functions $\mu_j^\pm(x)$, $j=1,2,\dots$ are Bohr almost periodic
\footnote{
For informations about almost periodic functions we refer to \cite{LZ}.}.
Using the trace formula (see for example \cite{L})
\beq\label{defp}
p_\pm(x)=E_0^\pm+\sum_{j=1}^\infty (E_{2j-1}^\pm+E_{2j}^\pm-2\mu_j^\pm(x)),
\eeq
we see that also $p_\pm(x)$ are real Bohr almost periodic. The operators 
\begin{equation}\label{Lpm}
  L_\pm:= -\frac{d^2}{dx^2}+p_\pm(x),\quad \text{dom}(L_\pm)=H^2(\R),
\end{equation}
in $L^2(\R)$, are then called almost periodic infinite-gap Schr\"odinger operators of the Levitan class. The spectra of $L_\pm$ are purely absolutely continuous and of the form
\beq\nn
\sigma_\pm=[E_0^\pm,E_1^\pm]\cup \dots\cup [E_{2j}^\pm,E_{2j+1}^\pm] \cup \dots,
\eeq
and have spectral properties analogous to the quasi-periodic finite-gap Schr\"odinger operator.
In particular, they are completely defined by the series $\sum_{j=1}^\infty (\mu_j^\pm,\sigma_j^\pm)$, which we call the Dirichlet divisor. 
These divisors are associated to Riemann surfaces of infinite genus, which are connected with the functions $Y_\pm^{1/2}(z)$, where 
\beq\label{defY}
Y_\pm(z)=-(z-E_0^\pm)\prod_{j=1}^\infty  \frac{(z-E_{2j-1}^\pm)}{E_{2j-1}^\pm}\frac{(z-E_{2j}^\pm)}{E_{2j-1}^\pm},
\eeq
and where the branch cuts are taken along the spectrum. 
It is known, that the Schr\"odinger equations 
\beq\label{eq:Weyl}
\Big(-\frac{d^2}{dx^2}+p_\pm(x)\Big)y(x)=z y(x)
\eeq
with any continuous, bounded potential $p_\pm(x)$ have two Weyl solutions $\psi_\pm(z,x)$ and $\breve\psi_\pm(z,x)$,
which satisfy 
 \beq\nn
\psi_\pm(z,.)\in L^2(\R_\pm), \quad \text{ resp. } \quad \breve\psi_\pm(z,.)\in L^2(\R_\mp),
\eeq
for $z\in\C\backslash\sigma_\pm$ and which are normalized by $\psi_\pm(z,0)=\breve\psi_\pm(z,0)=1$.  
In our case of Bohr almost periodic potentials of the Levitan class, these solutions have complementary properties similar to the properties of the Baker-Akhiezer functions in the finite-gap case. We will briefly discuss them in the next section. 

The object of interest, for us, is the one-dimensional Schr\"odinger operator $L$ in $L^2(\R)$ 
\beq\label{defL}
L:=-\frac{d^2}{dx^2}+q(x), \quad \text{dom}(L)=H^2(\R),
\eeq
with the real potential $q(x)\in C(\R)$ satisfying the following condition
\beq\label{moment}
\pm\int_0^{\pm\infty} (1+\vert x\vert^2 ) \vert q(x)-p_\pm(x)\vert dx<\infty,
\eeq
for which we will characterize the corresponding scattering data and derive the Gel'fand-Levitan-Marchenko equation with the help of the transformation operator, which has been investigated in \cite{G}.

\section{The Weyl solutions of the background operators}

In this section we want to summarize some facts for the background Schr\"odinger operators $L_\pm$ of Levitan class. We present these results, obtained in \cite{G}, \cite{L}, \cite{SY}, and \cite {SY2}, in a form, similar to the finite-gap case used in \cite{BET} and \cite{GH}.  

Let $L_\pm$ be the quasi-periodic one-dimensional Schr\"odinger operators associated with the potentials $p_\pm(x)$.
Let $s_\pm(z,x)$, $c_\pm(z,x)$ be sine- and cosine-type solutions of the corresponding equation 
\beq\label{eq:Lpm}
\left(-\frac{d^2}{dx^2}+p_\pm(x)\right)y(x)=zy(x), \quad z\in\C,
\eeq
associated with the initial conditions
\beq\nn
s_\pm(z,0)=c^\prime_\pm(z,0)=0, \quad c_\pm(z,0)=s^\prime_\pm(z,0)=1,
\eeq
where prime denotes the derivative with respect to $x$.
Then $c_\pm(z,x)$, $c^\prime_\pm(z,x)$, $s_\pm(z,x)$, and $s^\prime_\pm(z,x)$ are entire with respect to $z$. Moreover, they can be represented in the following form 
\begin{subequations}
 \begin{align*}
  c_\pm(z,x)&=\cos(\sqrt{z}x)+\int_0^x\frac{\sin(\sqrt{z}(x-y))}{\sqrt{z}}p_\pm(y)c_\pm(z,y)dy,\\ 
  s_\pm(z,x)&=\frac{\sin(\sqrt{z}x)}{\sqrt{z}}+\int_0^x \frac{\sin(\sqrt{z}(x-y))}{\sqrt{z}}p_\pm(y)s_\pm(z,y)dy
 \end{align*}
\end{subequations}

The background Weyl solutions are given by 
\begin{subequations} 
\begin{align*}
\psi_\pm(z,x) &= c_\pm(z,x) + m_\pm(z,0) s_\pm(z,x), \\ 
 \text{resp. } 
\breve{\psi}_\pm(z,x) & = c_\pm(z,x) +\breve{m}_\pm(z,0) s_\pm(z,x),
\end{align*}
\end{subequations}
where 
\begin{equation*}
m_\pm(z,x)=\frac{H_\pm(z,x)\pm Y_\pm^{1/2}(z)}{G_\pm(z,x)}, \quad \breve m_\pm(z,x)=\frac{H_\pm(z,x)\mp Y_\pm^{1/2}(z)}{G_\pm(z,x)},
\end{equation*}
are the Weyl functions of $L_\pm$ (cf \cite{L}), where $Y_\pm(z)$ are defined by \eqref{defY},
\beq\label{repH}
G_\pm(z,x)=\prod_{j=1}^\infty\frac{z-\mu_j^\pm(x)}{E_{2j-1}^\pm}, \quad \text{ and } \quad 
H_\pm(z,x)=\frac{1}{2}\frac{d}{d x}G_\pm(z,x).
\eeq 
Using (\ref{repmujx}) and (\ref{repH}), we have 
\begin{equation}\label{repHG}
H_\pm(z,x)=\frac{1}{2}\frac{d}{d x}G_\pm(z,x)= G_\pm(z,x) \sum_{j=1}^\infty \frac{\sigma_j^\pm(x) Y_\pm^{1/2}(\mu_j^\pm(x))}{\frac{d}{dz} G_\pm(\mu_j^\pm(x),x)(z-\mu_j^\pm(x))}.
\end{equation}
The Weyl functions $m_\pm(z,x)$ and $\breve m_\pm(z,x)$ are Bohr almost periodic.

\begin{lemma}\label{lempsi}
The background Weyl solutions, for $z\in\C$, can be represented in the following form 
\begin{align}\label{intrep}
\psi_{\pm}(z,x) =\exp\left(\int_0^x m_{\pm}(z,y)dy\right)= \left( \frac{G_\pm(z,x)}{G_\pm(z,0)}\right)^{1/2}\exp\left( \pm \int_0^x \frac{Y_\pm^{1/2}(z)}{G_\pm(z,y)}dy \right),
\end{align}
and 
\begin{align*}
 \breve\psi_\pm(z,x)=\exp\left(\int_0^x \breve m_\pm(z,y)dy\right)=\left(\frac{G_\pm(z,x)}{G_\pm(z,0)}\right)^{1/2}\exp\left(\mp \int_0^x \frac{Y_\pm^{1/2}(z)}{G_\pm(z,y)}dy\right).
\end{align*}
If for some $\varepsilon>0$, $\vert z-\mu_j^\pm(x)\vert > \varepsilon$ for all $j\in\N$ and $x\in\R$, then the following holds:
For any $1>\delta>0$ there exists an $R>0$ such that 
\begin{equation*}
 \vert\psi_{\pm}(z,x)\vert \leq \E^{\mp(1-\delta)x\Im(\sqrt{z})}\Big(1+\frac{D_R}{\vert z\vert}\Big), \text{ for any } \vert z\vert \geq R,\quad \pm x>0,  
\end{equation*}
and 
\begin{equation*}
 \vert \breve\psi_\pm(z,x)\vert \leq \E^{\pm(1-\delta)x\Im(\sqrt{z})}\Big(1+\frac{D_R}{\vert z\vert}\Big), \text { for any } \vert z\vert \geq R,\quad \pm x<0,
\end{equation*}
where $D_R$ denotes some constant dependent on $R$. 
\end{lemma}

As the spectra $\sigma_\pm$ consist of infinitely many bands, let us cut the complex plane along the spectrum $\sigma_\pm$ and
denote the upper and lower sides of the cuts by $\sipmu$ and
$\sipml$. The corresponding points on these cuts will be denoted by
$\lau$ and $\lal$, respectively. In particular, this means
\[
f(\lau) := \lim_{\varepsilon\downarrow0} f(\lambda+\I\varepsilon),
\qquad f(\lal) := \lim_{\varepsilon\downarrow0}
f(\lambda-\I\varepsilon), \qquad \lambda\in\sigma_\pm.
\]
Defining 
\begin{equation}\label{1.88}
g_\pm(\la)= -\frac{G_\pm(\la,0)}{2
Y_\pm^{1/2}(\la)},
\end{equation}
where the branch of the square root is chosen in such a way that
\begin{equation}\label{1.8}
\frac{1}{\I} g_\pm(\lau) = \Im(g_\pm(\lau))  >0 \quad
\mbox{for}\quad \lambda\in\sigma_\pm,
\end{equation}
it follows from Lemma~\ref{lempsi} that
\beq\label{1.62}
W(\breve\psi_\pm(z), \psi_\pm(z))=m_\pm(z)-\breve m_\pm(z)=\mp g_\pm(z)^{-1},
\eeq
where $W(f,g)(x)=f(x)g^\prime(x)-f^\prime(x)g(x)$ denotes the usual Wronskian determinant.

For every Dirichlet eigenvalue $\mu_j^\pm=\mu_j^\pm(0)$, the Weyl functions $m_\pm(z)$ and $\breve m_\pm(z)$ might have poles. 
If $\mu^\pm_j$ is in the interior of its gap, precisely one Weyl function $m_\pm$ or $\breve m_\pm$
will have a simple pole. Otherwise, if $\mu^\pm_j$ sits at an edge, both will have
a square root singularity. Hence we divide the set of poles accordingly:
\begin{align*}
M_\pm &=\{ \mu^\pm_j\mid\mu^\pm_j \in (E_{2j-1}^\pm,E_{2j}^\pm) \text{ and } m_\pm \text{ has a simple pole}\},\\
\breve M_\pm &=\{ \mu^\pm_j\mid\mu^\pm_j \in (E_{2j-1}^\pm,E_{2j}^\pm) \text{ and } \breve m_\pm \text{ has a simple pole}\},\\
\hat M_\pm &=\{ \mu^\pm_j\mid\mu^\pm_j \in \{E_{2j-1}^\pm,E_{2j}^\pm\} \},
\end{align*}
and we set $M_{r,\pm}=M_\pm\cup \breve M_\pm \cup\hat M_\pm$.

In particular, we obtain the following properties of the Weyl solutions (see, e.g. \cite{CL}, \cite{G}, \cite{L}, \cite{Te}, and \cite{T}):

\begin{lemma}\label{lem1.1}
The Weyl solutions have the following properties:
\begin{enumerate}[\rm(i)]
\item
The function $\psi_{\pm}(z,x)$ (resp. $\breve\psi_\pm(z,x)$) is holomorphic
as a function of $z$ in the domain $\mathbb{C}\setminus (\sigma_\pm\cup M_{\pm})$ (resp. $\C\backslash(\sigma_\pm\cup\breve M_\pm)$),
real valued on the set $\mathbb{R}\setminus \sigma_\pm$, and have
simple poles at the points of the set $M_{\pm}$ (resp. $\breve M_\pm$).
Moreover, they are  continuous up to the boundary $\sigma_\pm^u\cup \sigma_\pm^l$ except at the points from $\hat M_\pm$ and
\begin{equation}\label{1.10}
\psi_\pm(\lau) =\breve \psi_\pm(\lal) =\overline{\psi_\pm(\lal)},
\quad \lambda\in\sigma_\pm.
\end{equation}
For $E \in \hat M_\pm$ the Weyl solutions satisfy
\[
\psi_{\pm}(z,x)=O\left(\frac{1}{\sqrt{z-E}}\right), \quad \breve\psi_{\pm}(z,x)=O\left(\frac{1}{\sqrt{z-E}}\right),\quad
\mbox{as } z\to E\in \hat M_\pm,
\]
where the $O((z-E)^{-1/2})$-term is independent of $x$.\\
\noindent
The same applies to $\psi'_{\pm}(z,x)$ and $\breve\psi ^\prime_{\pm}(z,x)$.
\item
At the edges of the spectrum the Weyl solutions satisfy 
\beq\nn
\lim_{z\to E}\psi_\pm(z,x)-\breve\psi_\pm(z,x)=0 \quad \text{ for }\quad E\in\partial\sigma_\pm\backslash\hat M_\pm,
\eeq 
and 
\beq\nn
\psi_\pm(z,x)+\breve\psi_\pm(z,x)=O(1) \quad \text{ for } \quad z \text{ near }E\in \hat M_\pm,
\eeq 
where the $O(1)$-term depends on $x$.
\item
The functions $\psi_{\pm}(z,x)$ and $\breve\psi_{\pm}(z,x)$ form an orthonormal basis on the spectrum with respect to the weight 
\begin{equation}\label{1.12}
d\rho_\pm(z)=\frac{1}{2\pi\I}g_\pm(z) dz,
\end{equation}
and any $f(x)\in L^2(\R)$ can be expressed through
\beq\label{1.14}
f(x)=\oint_{\sigma_\pm} \left(\int_\R f(y)\psi_\pm(z,y)dy\right)\breve\psi_\pm(z,x)d\rho(z).
\eeq
Here we use the notation 
\beq\nn
\oint_{\sigma_\pm}f(z)d\rho_\pm(z) := \int_{\sigma^u_\pm} f(z)d\rho_\pm(z)
- \int_{\sigma^l_\pm} f(z)d\rho_\pm(z).
\eeq
\end{enumerate}
\end{lemma}

\begin{proof}
For a proof of (i) and (iii) we refer to \cite[Lemma 2.2]{G}.\\
(ii) We only prove the claim for the $+$ case (the $-$ case can be handled in the same way) and drop the $+$ in what follows. In \cite[Lemma 2.2]{G} we showed that 
\beq\label{intev}
\lim_{z\to E}\exp\left(\int_0^x \frac{Y^{1/2}(z)}{G(z,\tau)}d\tau\right)=\begin{cases}
 \pm 1, & \mu_j(0)\not =E, \mu_j(x)\not=E, \\
\pm 1, & \mu_j(0)=E, \mu_j(x)=E, \\
\pm\I, & \mu_j(0)=E, \mu_j(x)\not=E, \\
\pm\I, & \mu_j(0)\not=E, \mu_j(x)=E, \end{cases}
\eeq
for any $E\in\partial \sigma$.\\
Thus assuming that $\mu_j(0)\not =E$, we can write along the spectrum 
\begin{equation*}
 \psi(z,x)-\breve\psi(z,x)=2i\left(\frac{G(z,x)}{G(z,0)}\right)^{1/2}\sin\left(\int_0^x \frac{Y^{1/2}(z)}{G(z,\tau)}d\tau\right). 
\end{equation*}
\begin{enumerate}
 \item 
If in addition $\mu_j(x)\not= E$, then by \eqref{intev} we have $\lim_{z\to E}\sin\left(\int_0^x \frac{Y^{1/2}(z)}{G(z,\tau)}d\tau\right)=0$ and $\lim_{z\to E}\frac{G(z,x)}{G(z,0)}$ exists. Thus we end up with 
\begin{equation*}
 \lim_{z\to E} (\psi(z,x)-\breve\psi(z,x)) =0.
\end{equation*}
\item
If $\mu_j(x)=E$, then by \eqref{intev} we get $\lim_{z\to E}\sin\left(\int_0^x \frac{Y^{1/2}(z)}{G(z,\tau)}d\tau\right)=\pm 1$ and $\lim_{z\to E} \frac{G(z,x)}{G(z,0)}=0$. Hence
\begin{equation*}
 \lim_{z\to E}(\psi(z,x)-\breve\psi(z,x))=0.
\end{equation*}
\end{enumerate}

To prove the second claim, assume that $\mu_j(0)=E$ and write 
\begin{equation}\label{eq:lem22}
 \psi(z,x)-\breve\psi(z,x)=2\left(\frac{G(z,x)}{G(z,0)}\right)^{1/2}\cos\left(\int_0^x \frac{Y^{1/2}(z)}{G(z,\tau)}d\tau \right).
\end{equation}
\begin{enumerate}
 \item 
If $\mu_j(x)=E$, then by \eqref{intev} we get $\lim_{z\to E}\cos\left(\int_0^x\frac{Y^{1/2}(z)}{G(z,\tau)}d\tau\right)=\pm 1$ and $\lim_{z\to E} \frac{G(z,x)}{G(z,0)}$ exists. Therefore $\lim_{z\to E}(\psi(z,x)-\breve\psi(z,x))$ exists and especially 
\begin{equation*}
 \psi(z,x)-\breve\psi(z,x)=O(1) \quad \text{ for }\quad z\text{ near } E=\mu_j(0).
\end{equation*}
\item 
If $\mu_j(x)\not =E$, we cannot conclude as before, because $\lim_{z\to E}\frac{G(z,x)}{G(z,0)}$ does not exist. 
Assume that $E=E_{2j}$ (the case $E=E_{2j-1}$ can be handled in a similar way). Then we can seperate for fixed $x\in\R$ the interval $[0,x]$ into smaller intervals $[x_0,x_1]\cup [x_1,x_2]\cup \dots\cup [x_{2l},x]$ such that $x_0=0$, $\mu_j(x_k)\in\{E_{2j-1}, E_{2j}\}$ for $k=0,1,2, \dots ,2l$, $\mu_j(x)$ is monotone increading or decreasing on every interval $[x_k, x_{k+1}]$ and $\mu_j(x_0)=E_{2j}=\mu_j(x_{2l})$. Folowing the proof of \cite[Lemma 2.2 (i)]{G}, one obtains for $E=E_{2j}$ that 
\begin{align*}
 \int_0^x \frac{Y^{1/2}(z)}{G(z,\tau)}d\tau& =\I\sigma_j2(l+1) \arctan\left(\frac{\sqrt{E_{2j}-E_{2j-1}}}{\sqrt{z-E_{2j}}}\right)\\ \nn 
&\quad +\I\sigma_j\arctan\left(\frac{\sqrt{E_{2j}-\mu_j(x)}}{\sqrt{z-E_{2j}}}\right)+\I O(\sqrt{z-E_{2j}}), 
\end{align*}
where the $O(\sqrt{z-E_{2j}})$ term depends on $x$.
Using now that $\arctan (x)=\frac{\pi}{2}+O\left(\frac{1}{x}\right)$ for $x\to\infty$, we have 
\begin{equation*}
 \arctan\left(\frac{\sqrt{E_{2j}-E_{2j-1}}}{\sqrt{z-E_{2j}}}\right)=\frac{\pi}{2}+O\left(\frac{\sqrt{z-E_{2j}}}{\sqrt{E_{2j}-E_{2j-1}}}\right),
\end{equation*}
and 
\begin{equation*}
 \arctan\left(\frac{\sqrt{E_{2j}-\mu_j(x)}}{\sqrt{z-E_{2j}}}\right)=\frac{\pi}{2}+O\left(\frac{\sqrt{z-E_{2j}}}{\sqrt{E_{2j}-\mu_j(x)}}\right), 
\end{equation*}
which gives
\begin{equation*}
\cos\left(\int_0^x \frac{Y^{1/2}(z)}{G(z,\tau)}d\tau \right)=O(\sqrt{z-E_{2j}}),
\end{equation*}
where the $O(\sqrt{z-E_{2j}})$ term depends on $x$.
Plugging this into \eqref{eq:lem22} yields
\begin{equation*}
 \psi(z,x)-\breve\psi(z,x)=O(1)\quad \text{ for } \quad z\text{ near } E=\mu_j(0),
\end{equation*}
where the $O(1)$-term depends on $x$.
\end{enumerate}
\end{proof}

\section{The direct scattering problem}

Consider the Schr\"odinger equation 
\beq\label{eq:L}
\left(-\frac{d^2}{dx^2}+q(x)\right) y(x)=zy(x), \quad z\in\C,
\eeq
with a potential $q(x)$ satisfying the following condition 
\beq\label{decay}
\pm\int_0^{\pm\infty} (1+x^2)\vert q(x)-p_\pm(x)\vert dx <\infty.
\eeq
Then there exist two solutions, the so-called Jost solutions $\phi_\pm(z,x)$, which are asymptotically close to the background Weyl solutions $\psi_\pm(z,x)$ of equation (\ref{eq:Lpm}) as $x\to\pm\infty$ and they can be represented as
\beq\label{repphi}
 \phi_\pm(z,x)=\psi_\pm(z,x)\pm\int_x^{\pm\infty} K_\pm(x,y)\psi_\pm(z,y)dy.
\eeq
Here $K_\pm(x,y)$ are real-valued  functions, which are continuously differentiable with respect to both parameters and satisfy the estimate
\beq\label{estK}
\vert K_\pm(x,y) \vert \leq C_\pm(x)Q_\pm(x+y)=\pm C_\pm(x)\int_{\frac{x+y}{2}}^{\pm\infty}\vert q(t)-p_\pm(t)\vert dt,
\eeq
where $C_\pm(x)$ are continuous, positive, monotonically decreasing functions, and therefore bounded as $x\to\pm\infty$.
Furthermore,
\beq\label{estpK}
\left\vert\frac{d K_\pm(x,y)}{dx}\right\vert+\left\vert \frac{d K_\pm(x,y)}{dy}\right\vert \leq C_\pm(x)\left(\left\vert q_\pm\left(\frac{x+y}{2}\right)\right\vert +Q_\pm(x+y)\right)
\eeq
and 
\beq\label{Kxx}
\pm\int_a^{\pm\infty}(1+x^2)\left\vert \frac{d}{dx}K_\pm(x,x)\right\vert dx<\infty, \quad \forall a\in\R.
\eeq
For more information we refer to \cite{G}.

Moreover, for $\lambda \in\sigma_\pm^u\cup\sigma_\pm^l$ a second pair of solutions of (\ref{eq:L}) is given by
\beq\label{repolphi}
\overline{\phi_\pm(\la,x)}=\breve\psi_\pm(\la,x)\pm\int_x^{\pm\infty}K_\pm(x,y)\breve\psi_\pm(\la,y)dy, \quad \la\in\sigma_\pm^u\cup\sigma_\pm^l.
\eeq
Note $\breve\psi_\pm(\la,x)=\overline{\psi_\pm(\la,x)}$ for $\la\in\sigma_\pm$.

Unlike the Jost solutions $\phi_\pm(z,x)$, these solutions only exist on the upper and lower cuts of the spectrum and cannot be continued to the whole complex plane. Combining (\ref{1.62}), (\ref{repphi}), (\ref{estK}), and (\ref{repolphi}), one obtains 
\beq\label{Wpop}
W(\phi_\pm(\la),\overline{\phi_\pm(\la)})=\pm g(\la)^{-1}.
\eeq

In the next lemma we want to point out, which properties of the background Weyl solutions are also inherited by the Jost solutions.

\begin{lemma}\label{propphi}
The Jost solutions $\phi_\pm(z,x)$ have the following properties:
\begin{enumerate}
\item
The function $\phi_\pm(z,x)$ considered as a function of $z$, is holomorphic in the domain $\C\backslash(\sigma_\pm \cup M_\pm)$, and has simple poles at the points of the set $M_\pm$. It is continuous up to the boundary $\sigma_\pm^u \cup\sigma_\pm^l$ except at the points from $\hat{M}_\pm$. Moreover, we have 
\beq\nn
\phi_\pm(z,x)\in L^2(\R_\pm), \quad z\in\C\backslash\sigma_\pm
\eeq
For $E\in \hat{M}_\pm$ they satisfy 
\beq\nn
\phi_\pm(z,x)=O\left(\frac{1}{\sqrt{z-E}}\right), \quad \text{ as } z\to E\in\hat{M}_\pm,
\eeq
where the $O((z-E)^{-1/2})$-term depends on $x$.
\item
At the band edges of the spectrum we have the following behavior:
\begin{equation*}
\lim_{z\to E}\phi_\pm(z,x)-\overline{\phi_\pm(z,x)}=0\quad \text{ for }  \quad E\in\partial\sigma_\pm\backslash\hat{M}_\pm,
\end{equation*} 
and 
\begin{equation*}
\phi_\pm(z,x)+\overline{\phi_\pm(z,x)}=O(1) \quad  \text{ for } \quad z \text{ near }E\in\hat{M}_\pm, 
\end{equation*}
where the $O(1)$-term depends on $x$.
\end{enumerate}
\end{lemma}

\begin{proof}
Everything follows from the fact that these properties are only dependent on $z$ and therefore the transformation operator does not influence them.
\end{proof}

Now we want to characterize the spectrum of our operator $L$, which consists of an (absolutely) continuous part, $\sigma=\sigma_+\cup\sigma_-$ and an at most countable number of discrete eigenvalues, which are situated in the gaps, $\sigma_d\subset\R\backslash\sigma$. In particular every gap can only contain a finite number of discrete eigenvalues (cf. \cite{kr}, \cite{krt1}, and \cite[Thm. 6.12]{RK}) and thus they cannot cluster. For our purposes it will be convenient to write 
\begin{equation*}
\sigma=\sigma_-^{(1)}\cup\sigma_+^{(1)}\cup\sigma^{(2)},
\end{equation*}
with
\begin{equation*}
 \sigma^{(2)}:=\sigma_-\cap\sigma_+, \quad \sigma_\pm^{(1)}=\clos(\sigma_\pm\backslash\sigma^{(2)}).
\end{equation*}

It is well-known that a point $\lambda\in\R\backslash\sigma$ corresponds to the discrete spectrum if and only if the two Jost solutions are linearly dependent, which implies that we should investigate 
\beq
W(z):=W(\phi_-(z,.), \phi_+(z,.)),
\eeq
the Wronskian of the Jost solutions.
This is a meromorphic function in the domain $\C\backslash\sigma$, with possible poles at the points $M_+\cup M_- \cup (\hat{M}_+\cap \hat{M}_-)$ and possible square root singularities at the points $(\hat{M}_+\cup \hat{M}_-)\backslash (\hat{M}_+\cap\hat{M}_-)$.
For investigating the function $W(z)$ in more detail, we will multiply the possible poles and square root singularities away.
Thus we define locally in a small neighborhood $U_j^\pm$ of the j'th gap $[E_{2j-1}^\pm,E_{2j}^\pm]$, where $j=1,2,\dots$ 
\beq
\tilde{\phi}_{j,\pm}(z,x)=\delta_{j,\pm}(z)\phi_\pm(z,x),
\eeq
where
\begin{equation}
\delta_{j,\pm}(z)=\begin{cases}
 z-\mu_j^\pm, & \text{ if } \mu_j^\pm\in M_\pm,\\ 
 1, & \text{else}\end{cases}
\end{equation}
and 
\beq
\hat{\phi}_{j,\pm}(z,x)=\hat{\delta}_{j,\pm}(z)\phi_\pm(z,x),
\eeq
where
\beq
\hat{\delta}_{j,\pm}(z)=\begin{cases}
z-\mu_j^\pm, & \text{ if } \mu_j^\pm\in M_\pm,\\
\sqrt{z-\mu_j^\pm}, & \text{ if } \mu_j^\pm\in\hat{M}_\pm,\\
1, & \text{ else}. \end{cases}
\eeq
Correspondingly, we set
\beq\label{deftihW}
\tilde{W}(z)=W(\tilde{\phi}_-(z,.),\tilde{\phi}_+(z,.)), \quad \hat{W}(z)=W(\hat{\phi}_-(z,.),\hat{\phi}_+(z,.)).
\eeq

Here we use the definitions 
\beq\label{deftiphi}
\tilde{\phi}_\pm(z,x)=\begin{cases}
\tilde\phi_{j,\pm}(z,x), & \text{ for } z\in U_j^\pm, j=1,2,\dots,\\
\phi_\pm(z,x), & \text{ else },\end{cases}
\eeq
\beq
\hat{\phi}_\pm(z,x)=\begin{cases} 
\hat{\phi}_{j,\pm}(z,x), & \text{ for } z\in U_j^\pm, j=1,2,\dots,\\
\phi_\pm(z,x), & \text{ else }. \end{cases}
\eeq
and we will choose $U_j^+=U_m^-$, if $[E_{2j-1}^+, E_{2j}^+]\cap [E_{2m-1}^-,E_{2m}^-]\not = \emptyset$.
Analogously, one can define $\delta_\pm(z)$ and $\hat{\delta}_\pm(z)$.

Note that the function $ \hat{W}(z)$ is holomorphic in the domain $U_j^\pm \cap (\C\backslash\sigma)$ and continuous up to the boundary. But unlike the functions $W(z)$ and $\tilde{W}(z)$ it may not take real values on the set $\R\backslash \sigma$ and complex conjugated values on the different sides of the spectrum $\sigma^u \cup \sigma^l$ inside the domains $U_j^\pm$. That is why we will characterize the spectral properties of our operator $L$ in terms of the function $\tilde{W}(z)$ which can have poles at the band edges.

Since the discrete spectrum of our operator $L$ is at most countable, we can write it as 
\begin{equation*}
\sigma_d=\bigcup_{n=1}^\infty \sigma_n\subset \R\backslash\sigma,
\end{equation*}
where 
\begin{equation*}
\sigma_n=\{\la_{n,1}, \dots ,\la_{n,k(n)}\}, \quad n\in\N
\end{equation*}
and $k(n)$ denotes the number of eigenvalues in the n'th gap of $\sigma$.

For every eigenvalue $\la_{n,m}$ we can introduce the corresponding norming constants 
\beq\label{defga}
(\gamma_{n,m}^\pm)^{-2}=\int_{\R}\tilde{\phi}_{\pm}^2(\la_{n,m},x)dx.
\eeq
Now we begin with the study of the properties of the scattering data. Therefore we introduce the scattering relations
\beq \label{scatrel}
T_\pm(\la)\phi_\mp(\la,x)=\overline{\phi_\pm(\la,x)}+R_\pm(\la)\phi_\pm(\la,x), \quad \la\in\sigma_\pm^{u,l},
\eeq
where the transmission and reflection coefficients are defined as usual,
\beq\label{defTR}
T_\pm(\la):=\frac{W(\overline{\phi_\pm(\la)},\phi_\pm(\la))}{W(\phi_\mp(\la),\phi_\pm(\la))}, \quad 
R_\pm(\la):=-\frac{W(\phi_\mp(\la),\overline{\phi_\pm(\la))}}{W(\phi_\mp(\la),\phi_\pm(\la))}, \quad \la\in\sigma_\pm^{u,l}
\eeq

\begin{theorem}\label{lem:scat}
For the scattering matrix the following properties are valid:
\begin{enumerate}
\item
$T_\pm(\la^u)=\overline{T_\pm(\la^l)}$ and $R_\pm(\la^u)=\overline{R_\pm(\la^l)}$ for $\la\in\sigma_\pm$.
\item
$\dfrac{T_\pm(\la)}{\overline{T_\pm(\la)}}=R_\pm(\la)$ for $\la\in\sigma_\pm^{(1)}$.
\item
$1-\vert R_\pm(\la)\vert^2=\dfrac{g_\pm(\la)}{g_\mp(\la)}\vert T_\pm(\la)\vert^2$ for $\la\in\sigma^{(2)}$.
\item
$\overline{R_\pm(\la)}T_\pm(\la)+R_\mp(\la)\overline{T_\pm(\la)}=0$ for $\la\in\sigma^{(2)}$.
\end{enumerate}
\end{theorem}

\begin{proof}
(i) and (iv) follow from (\ref{repphi}), (\ref{repolphi}), (\ref{defTR}), and Lemma~\ref{lem1.1} \\
For showing (ii) observe that $\tilde{\phi}_\mp(\la,x)\in\R$ as $\la\in\inte(\sigma_\pm^{(1)})$, which implies (ii).\\
To show (iii), assume $\la\in\inte\sigma^{(2)}$, then by (\ref{scatrel}) 
\beq\nn
\vert T_\pm\vert^2W(\phi_\mp,\overline{\phi_\mp})=(\vert R_\pm \vert^2-1)W(\phi_\pm,\overline{\phi_\pm}).
\eeq
Thus using (\ref{Wpop}) finishes the proof.
\end{proof}

\begin{theorem}\label{thm:asym}
The transmission and reflection coefficients have the following asymptotic behavior, as $\la\to\infty$ for $\la\in\sigma^{(2)}$ outside a small $\varepsilon$ neighborhood of the band edges of $\sigma^{(2)}$:
\begin{align*}
R_\pm(\la)&=O(\vert \la \vert^{-1/2}),\\
T_\pm(\la)&=1+O(\vert \la \vert^{-1/2}).
\end{align*}
\end{theorem}

\begin{proof}
The asymptotics can only be valid for $\la\in\sigma^{(2)}$ outside an $\varepsilon$ neighborhood of the band edges, because the Jost solutions $\phi_\pm$ might have square root singularities there. At first we will investigate $W(\phi_-(\la,0),\phi_+(\la,0))$:
\begin{align}
\phi_-(\la,0)\phi_+^\prime(\la,0)=& \left(1+\int_{-\infty}^0 K_-(0,y)\psi_-(\la,y)dy\right)\\ \nn
& \times \left(m_+(\la)-K_+(0,0)+\int_0^\infty K_{+,x}(0,y)\psi_+(\la,y)dy\right).
\end{align}
Using (cf. (\ref{intrep})) $$\psi_\pm^\prime(\la,x)=m_\pm(\la,x)\psi_\pm(\la,x),$$
we can write 
\begin{equation*}
\int_{-\infty}^0 K_-(0,y)\psi_-(\la,y)dy
= \int_{-\infty}^0 \frac{K_-(0,y)}{m_-(\lambda,y)}\psi_-^\prime(\lambda, y)dy.
\end{equation*}
Hence
\beq\nn
\int_{-\infty}^0 K_-(0,y)\psi_-(\la,y)dy=\frac{K_-(0,0)}{m_-(\la)}+I_1(\la),
\eeq
\begin{equation*}
I_1(\la)=-\int_{-\infty}^0\left(K_{-,y}(0,y)\frac{\psi_-(\la,y)}{m_-(\la,y)}-K_-(0,y)\psi_-(\la,y)\frac{m_-^\prime(\la,y)}{m_-(\la,y)^2}\right)dy.
\end{equation*}
Here it should be noticed that $m_\pm(\la,y)^{-1}$ has no pole, because (see e.g. \cite{L}) 
\begin{equation*}
G_\pm(\la,y)N_\pm(\la,y)+H_\pm(\la,y)^2=Y_\pm(\la), 
\end{equation*}
where 
\begin{equation*}
N_\pm(\la,y)=-(\la-\nu_0^\pm(y))\prod_{j=1}^\infty\frac{\la-\nu_j^\pm(y)}{E_{2j-1}^\pm},
\end{equation*}
with $\nu_0^\pm(y)\in(-\infty,E_0^\pm]$ and $\nu_j^\pm(y)\in[E_{2j-1}^\pm,E_{2j}^\pm]$.
Thus we obtain
\begin{equation*}
m_\pm(\la,y)^{-1}=\frac{G_\pm(\la,y)}{H_\pm(\la,y)\pm Y_\pm(\la)^{1/2}}=-\frac{H_\pm(\la,y)\mp Y_\pm(\la)^{1/2}}{N_\pm(\la,y)},
\end{equation*}
and therefore $\frac{K_-(0,0)}{m_-(\la)}=O(\frac{1}{\sqrt{\lambda}})$.

Moreover $I_1(\la)=O\big(\frac{1}{\sqrt{\la}}\big)$ as the following estimates show:
\begin{align*}
\vert I_1(\la)& \vert \leq \int_{-\infty}^0 \vert K_{-,y}(0,y)\frac{\psi_-(\la,y)}{m_-(\la,y)}\vert dy +\int_{-\infty}^0\vert K_-(0,y)\psi_-(\la,y)\frac{m_-^\prime(\la,y)}{m_-(\la,y)^2}\vert dy \\ \nn
&\leq \frac{C}{\sqrt{\la}}\int_{-\infty}^0(\vert q(y)-p_-(y) \vert +Q_-(y))dy,
\end{align*}
where we used that $\vert \psi_\pm(\la,y)\vert =\vert \frac{G _\pm(\la,y)}{G_\pm(\la,0)}\vert=O(1)$ and  $m_\pm^{-1}(\la,y)=O\left(\frac{1}{ \sqrt{\la}}\right)$ for all $y$ by the quasi-periodicity, together with \eqref{eq:Weyl} and 
$$\psi_\pm^{\prime\prime}(\la,x)=m_\pm(\la,x)^2\psi_\pm(\la,x)+m_\pm^\prime(\la,x)\psi_\pm(\la,x).$$
Making the same conclusions as before, one obtains
\beq\nn
\int_0^\infty K_{+,x}(0,y)\psi_+(\la,y)dy=O(1).
\eeq

In a similar manner one can investigate 
\begin{align}\nn
\phi_-^\prime(\la,0)\phi_+(\la,0)=&\left(m_-(\la)+K_-(0,0)+\int_{-\infty}^0 K_{-,x}(0,y)\psi_-(\la,y)dy\right) \\
&\times\left(1+\int_0^\infty K_+(0,y)\psi_+(\la,y)dy\right),
\end{align}
where 
\begin{equation*}
\int_{-\infty}^0 K_{-,x}(0,y)\psi_-(\la,y)dy=O(1),
\end{equation*}
\begin{equation*}
\int_0^\infty K_+(0,y)\psi_+(\la,y)dy=-\frac{K_+(0,0)}{m_+(\la)}+I_2(\la),
\end{equation*}
\begin{equation*}
I_2(\la)=-\int_0^\infty \left(K_{+,y}(0,y)\frac{\psi_+(\la,y)}{m_+(\la,y)}-K_+(0,y)\psi_+(\la,y)\frac{m_+^\prime(\la,y)}{m_+(\la,y)^2}\right)dy,
\end{equation*}
and $I_2(\la)=O\big(\frac{1}{\sqrt{\la}}\big)$.
Thus combining all the informations we obtained so far yields
\begin{align}\nn
W(\phi_-(\la),\phi_+(\la))& = m_+(\la)-m_-(\la)+ K_-(0,0)\left(\frac{m_+(\la)-m_-(\la)}{m_-(\la)}\right)\\ \nn
&+ K_+(0,0)\left(\frac{m_-(\la)-m_+(\la)}{m_+(\la)}\right)+O(1).
\end{align}
and therefore, using \eqref{Wpop}, 
\begin{equation*}
T_\pm(\la)=1+O\left(\frac{1}{\sqrt{\la}}\right).
\end{equation*}
Analogously one can investigate the behavior of $W(\phi_\mp(\la),\overline{\phi_\pm(\la)}$ to obtain $R_\pm(\la)=O\Big(\frac{1}{\sqrt{\la}}\Big)$.
\end{proof}

\begin{theorem}\label{thm:TW}
The functions $T_\pm(\la)$ can be extended analytically to the domain $\C\backslash (\sigma \cup M_\pm \cup \breve{M}_\pm)$ and satisfy
\beq\label{eq:TgW}
\frac{-1}{T_+(z)g_+(z)}=\frac{-1}{T_-(z)g_-(z)}=:W(z),
\eeq
where $W(z)$ possesses the following properties:
\begin{enumerate}
\item
The function $\tilde{W}$ is holomorphic in the domain $U_j^\pm \cap (\C\backslash\sigma)$, with simple zeros at the points $\la_k$, where 
\beq\label{dertiW}
\left( \frac{d\tilde{W}}{dz}(\la_k)\right)^2=(\gamma_{n,k}^+ \gamma_{n,k}^-)^{-2}.
\eeq
Besides it satisfies 
\beq\label{tiWreal}
\overline{\tilde{W}(\la^u)}=\tilde{W}(\la^l), \quad \la\in U_j^\pm \cap\sigma \quad \text{and} \quad \tilde{W}(\la)\in\R , \quad \la\in U_j^\pm \cap(\R\backslash\sigma).
\eeq
\item
The function $\hat{W}(z)$ is continuous on the set $U_j^\pm \cap \C\backslash\sigma$ up to the boundary $\sigma^l\cup\sigma^u$. It can have zeros on the set $\partial\sigma \cup (\partial\sigma_+^{(1)} \cap\partial\sigma_-^{(1)}) $ and does not vanish at any other points of $\sigma$. If $\hat{W}(E)=0$ as $E\in\partial\sigma \cup (\partial\sigma_+^{(1)}\cap\partial\sigma_-^{(1)})$, then $\hat{W}(z)=\sqrt{z-E}(C(E)+o(1))$, $C(E)\not =0$.
\end{enumerate}
\end{theorem}

\begin{proof}
\begin{enumerate}
\item
Except for (\ref{dertiW}) everything follows from the corresponding properties of $\phi_\pm(z,x)$. Therefore assume $\hat{W}(\la_0)=0$ for some $\la_0\in\C\backslash\sigma$, then
\beq\label{lidep}
\tilde{\phi}_\pm(\la_0,x)=c_\pm\tilde{\phi}_\mp(\la_0,x),
\eeq
for some constants $c_\pm$, which satisfy $c_-c_+=1$. Moreover, every zero of $\tilde{W}$ (or $\hat{W}$) outside the continuous spectrum, is a point of the discrete spectrum of $L$ and vice versa.

Denote by $\gamma_\pm$ the corresponding norming constants defined in (\ref{defga}) for some fixed point $\la_0$ of the discrete spectrum. 
Proceeding as in \cite{M} one obtains 
\beq\label{intW}
W\big(\tilde{\phi}_\pm(\la_0,0), \frac{d}{d\la}\tilde{\phi}_\pm(\la_0,0)\big)=\int_0^{\pm\infty} \tilde{\phi}_\pm^2(\la_0,x)dx.
\eeq
Thus using (\ref{lidep}) and (\ref{intW}) yields
\begin{align}\label{cga}
\gamma_\pm^{-2}& =\mp c_\pm^2\int_0^{\mp\infty}\tilde{\phi}_\mp^2(\la_0,x)dx \pm\int_0^{\pm\infty}\tilde{\phi}_\pm^2(\la_0,x)dx\\ \nn
& =\mp c_\pm^2 W\big( \tilde{\phi}_\mp(\la_0,0),\frac{d}{d\la}\tilde{\phi}_\mp(\la_0,0)\big)\pm W\big(\tilde{\phi}_\pm(\la_0,0), \frac{d}{d\la}\tilde{\phi}_\pm(\la_0,0)\big) \\ \nn 
& =c_\pm\frac{d}{d\la}W(\tilde{\phi}_-(\la_0),\tilde{\phi}_+(\la_0)).
\end{align}
Applying now $c_-c_+=1$, we obtain (\ref{dertiW}).
\item
The continuity of $\hat{W}(z)$ up to the boundary follows immediately from the corresponding properties of $\hat{\phi}_\pm(z,x)$. Now we will investigate the possible zeros of $\hat W(\lambda)$ for $\lambda\in\sigma$.

Assume $W(\la_0)=0$ for some $\la_0\in\inte(\sigma^{(2)})$. Then $\phi_+(\la_0,x)=c\phi_-(\la_0,x)$ and $\overline{\phi_+(\la_0,x)}=\overline{c}\overline{\phi_-(\la_0,x)}$. Thus $W(\phi_+,\overline{\phi_+})=\vert c\vert^2 W(\phi_-, \overline{\phi_-})$ and therefore $\sign g_+(\la_0)=-\sign g_-(\la_0)$ by (\ref{Wpop}), contradicting (\ref{1.8}).

Next let $\la_0\in\inte(\sigma_\pm^{(1)})$ and $\tilde{W}(\la_0)=0$, then $\phi_\pm(\la_0,x)$ and $\overline{\phi_\pm(\la_0,x)}$ are linearly independent and bounded, moreover $\tilde{\phi}_\mp(\la_0,x)\in\R$. Therefore $\tilde{W}(\la_0)=0$ implies that $\tilde{\phi}_\mp=c_1^\pm\phi_\pm=c_2^\pm\overline{\phi_\pm}$ and thus $W(\phi_\pm,\overline{\phi_\pm})=0$ , which is impossible by (\ref{Wpop}). Note that in this case $\la_0$ can coincide with a pole $\mu\in M_\mp$.

Since $\hat W(\lambda)\not =0$ for $\lambda\in \inte(\sigma^{(2)})\cup\inte(\sigma^{(1)}_+)\cup\inte(\sigma^{(1)}_-)$, it is left to investigate the behavior at the band edges of $\sigma_+$ and $\sigma_-$.
Therefore introduce the local parameter $\tau=\sqrt{z-E}$ in a small neighborhood of each point $E\in\partial\sigma_\pm$ and define $\dot{y}(z,x)=\frac{d}{d\tau}y(z,x)$. A simple calculation shows that $\frac{dz}{d\tau}(E)=0$, hence for every solution $y(z,x)$ of (\ref{eq:L}), its derivative $\dot{y}(E,x)$ is again a solution of (\ref{eq:L}). Therefore, the Wronskian $W(y(E),\dot{y}(E))$ is independent of $x$.

For each $x\in\R$ in a small neighborhood of a fixed point $E\in\partial\sigma_\pm$ we introduce the function 
\beq\label{hpsiE}
\hat{\psi}_{\pm,E}(z,x)=\begin{cases} 
\psi_\pm(z,x), &  E\in\partial\sigma_\pm\backslash\hat{M}_\pm, \\ \nn
\tau\psi_\pm(z,x), & E\in \hat{M}_\pm.\end{cases}
\eeq

Proceeding as in \cite[Lem. B.1]{BET} one obtains
\beq\label{Wdphi}
W\Big(\hat{\psi}_{\pm,E}(E), \frac{d}{d\tau}\hat{\psi}_{\pm,E}(E)\Big)=\pm\lim_{z\to E}\frac{\alpha \tau^\alpha}{2g_\pm(z)},
\eeq
where $\alpha=-1$ if $E\in\partial\sigma_\pm\backslash\hat{M}_\pm$ and $\alpha=1$ if $E\in\hat{M}_\pm$.

Using representation (\ref{intrep}) for $\psi_\pm(z,x)$ one can show (cf \cite{G}), 
\beq\nn
\psi_\pm(E,x)=\left(\frac{G_\pm(E,x)}{G_\pm(E,0)}\right)^{1/2}\exp\left(\pm\lim_{z\to E}\int_0^x\frac{Y_\pm(z)^{1/2}}{G_\pm(z,\tau)}d\tau\right), \quad E\in\partial\sigma
\eeq
where 
\beq\nn
\exp\left(\pm\lim_{z\to E}\int_0^x\frac{Y_\pm(z)^{1/2}}{G_\pm(z,\tau)}d\tau\right)=\begin{cases}
\I^{2s+1}, & \mu_j\not=E, \mu_j(x)=E,\\
\I^{2s+1}, &  \mu_j=E, \mu_j(x)\not=E,\\
\I^{2s}, & \mu_j=E, \mu_j(x)=E,\\
\I^{2s}, & \mu_j\not=E, \mu_j(x)\not =E,\end{cases}
\eeq
for $s\in\{0,1\}$.
Defining 
\beq\nn
\hat{\phi}_{\pm,E}(\la,x)=\begin{cases}
\phi_\pm(\la,x), & E\in\partial\sigma_\pm\backslash\hat{M}_\pm,\\
\tau\phi_\pm(\la,x), & E\in\hat{M}_\pm,\end{cases}
\eeq
we can conclude using (\ref{repphi}) that 
\beq\label{behE}
\overline{\phi_\pm(E,x)}=\phi_\pm(E,x), \quad \text{ for } E\in\partial\sigma_\pm\backslash\hat{M}_\pm.
\eeq
Moreover, for $E\in\hat{M}_\pm$,
\begin{equation}\label{behE2}
\begin{cases}
\overline{\hat{\phi}_{\pm,E}(E,x)}=-\hat{\phi}_{\pm,E}(E,x), & \text{ a left band edge from } \sigma_\pm,\\ 
\overline{\hat{\phi}_{\pm,E}(E,x)}=\hat{\phi}_{\pm,E}(E,x), & \text{ a right band edge from } \sigma_\pm.
\end{cases}
\end{equation}
If $\la_0=E\in\partial\sigma^{(2)}\cap \inte(\sigma_\pm)\subset \inte(\sigma_\pm)$, then $\hat{W}(E)=0$ if and only if $W(\psi_\pm,\hat{\psi}_{\mp,E})(E)=0$. 
Therefore, as $\hat{\phi}_{\mp,E}(E,.)$ are either pure real or pure imaginary, $W(\overline{\phi_\pm},\hat{\phi}_{\mp,E})(E)=0$, which implies that $\overline{\phi_\pm}(E,x)$ and $\phi_\pm(E,x)$ are linearly dependent, a contradiction.

Thus the function $\hat{W}(z)$ can only be zero at points $E$ of the set $\partial\sigma\cup (\partial\sigma_+^{(1)}\cap\partial\sigma_-^{(1)})$. We will now compute the order of the zero. First of all note that the function $\hat{W}(\la)$ is continuously differentiable with respect to the local parameter $\tau$. Since $\frac{d}{d\tau}(\delta_+\delta_-)(E)=0$, the function $W(\hat{\phi}_{+,E}, \hat{\phi}_{-,E})$ has the same order of zero at $E$ as $\hat{W}(\la)$. Moreover, if $\hat{\delta}_\pm(E)\not =0$,then $\frac{d}{d\tau}\hat{\delta}_\pm(E)=0$ and if $\hat{\delta}_-(E)=\hat{\delta}_+(E)=0$, then $\frac{d}{d\tau}(\tau^{-2}\hat{\delta}_+\hat{\delta}_-)(E)=0$. Hence $\frac{d}{d\tau}\hat{W}(E)=0$ if and only if $\frac{d}{d\tau}W(\hat{\phi}_{+,E}, \hat{\phi}_{-,E})=0$.

Combining now all the informations we obtained so far, we can conclude as follows:
if $\hat{W}(E)=0$, then $\hat{\phi}_{\pm,E}(E,.)=c_\pm \hat{\phi}_{\mp,E}(E,.)$, with $c_-c_+=1$. Furthermore we can write 
\begin{align}\nn 
\dot{W}(\hat{\phi}_{+,E},\hat{\phi}_{-,E})(E)& =W(\frac{d}{d\tau}\hat{\phi}_{+,E},\hat{\phi}_{-,E})(E)-W(\frac{d}{d\tau}\hat{\phi}_{-,E}, \hat{\phi}_{+,E})(E)\\ \nn
& = c_-W(\frac{d}{d\tau}\hat{\phi}_{+,E},\hat{\phi}_{+,E})(E)-c_+W(\frac{d}{d\tau}\hat{\phi}_{-,E}, \hat{\phi}_{-,E})(E)\\ \nn
& = c_-W(\frac{d}{d\tau}\hat{\psi}_{+,E},\hat{\psi}_{+,E})(E)-c_+W(\frac{d}{d\tau}\hat{\psi}_{-,E}, \hat{\psi}_{-,E})(E).
\end{align}
Using (\ref{Wdphi}), (\ref{behE}), (\ref{behE2}), and distinguishing several cases as in \cite[Lem. B.1]{BET} finishes the proof. 
\end{enumerate}
\end{proof}

\begin{theorem}\label{thm:pR}
The reflection coefficient $R_\pm(\la)$ satisfies:
\begin{enumerate}
\item 
The reflection coefficient $R_\pm(\la)$ is a continuous function on the set $\inte(\sigma_\pm^{u,l})$.
\item
If $E\in\partial\sigma_+\cap\partial\sigma_-$ and $\hat{W}(E)\not=0$, then the function $R_\pm(\la)$ is also continuous at $E$. Moreover,
\beq\nn
R_\pm(E)=\begin{cases}
-1 & \text{ for } E\not \in \hat{M}_\pm,\\
1 & \text{ for } E\in\hat{M}_\pm.
         \end{cases}
\eeq	
\end{enumerate}

\end{theorem}

\begin{proof}
\begin{enumerate}
\item
At first it should be noted that by Lemma~\ref{lem:scat} the reflection coefficient is bounded, as $\frac{g_\pm(\la)}{g_\mp(\la)}>0$ for $\la\in\inte(\sigma^{(2)})$. Thus, using the corresponding properties of $\phi_\pm(z,x)$, finishes the first part.
\item
We proceed as in the proof of \cite[Lemma 3.3 III.(b)]{BET}.
By (\ref{defTR}) the reflection coefficient can be represented in the following form:
\beq\label{defR}
R_\pm(\la)=-\frac{W(\overline{\phi_\pm(\la)},\phi_\mp(\la))}{W(\phi_\pm(\la),\phi_\mp(\la))}=\pm\frac{W(\overline{\phi_\pm(\la)},\phi_\mp(\la))}{W(\la)},
\eeq
and is therefore continuous on both sides of the set $\inte(\sigma_\pm)\backslash (M_\mp \cup \hat{M}_\mp)$. 
Moreover,
\beq\nn
\vert R_\pm(\la)\vert =\left\vert \frac{W(\overline{\hat{\phi}_\pm(\la)},\hat{\phi}_\mp(\la))}{\hat{W}(\la)}\right\vert,
\eeq
where the denominator does not vanish, by assumption and hence $R_\pm(\la)$ is continuous on both sides of the spectrum in a small neighborhood of the band edges under consideration.

Next, let $E\in\{E_{2j-1}^\pm, E_{2j}^\pm\}$ with $\hat{W}(E)\not=0$. Then, if $E\not\in\hat{M}_\pm$, we can write 
\beq\nn
R_\pm(\la)=-1\mp\frac{\hat{\delta}_{j,\pm}(\la)W(\phi_\pm(\la)-\overline{\phi_\pm(\la)},\hat{\phi}_\mp(\la))}{\hat{W}(\la)},
\eeq
which implies $R_\pm(\la)\to -1$, since $\phi_\pm(\la)-\overline{\phi_\pm(\la)}\to 0$ by Lemma~\ref{propphi} as $\la\to E$. Thus we proved the first case. \\
If $E\in\hat{M}_\pm$ with $\hat{W}(E)\not =0$, we use (\ref{defR}) in the form
\beq\nn
R_\pm(\la)=1\pm\frac{\hat{\delta}_{j,\pm}(\la)W(\phi_\pm(\la)+\overline{\phi}_\pm(\la),\hat{\phi}_\mp(\la))}{\hat{W}(\la)},
\eeq
which yields $R_\pm(\la)\to 1$, since $\hat{\delta}_{j,\pm}(\la)\to 0$ and $\phi_\pm(\la)+\overline{\phi_\pm}(\la)=O(1)$ by Lemma~\ref{propphi} as $\la\to E$. This settles the second case.
\end{enumerate}
\end{proof}

\section{The Gel'fand-Levitan-Marchenko Equation}
The aim of this section is to derive the Gel'fand-Levitan-Marchenko (GLM) equation, which is also called the inverse scattering problem equation and to obtain some additional properties of the scattering data, as a consequence of the GLM equation.

Therefore consider the function 
\begin{align*}
G_\pm(z,x,y) &= T_\pm(z)\phi_\mp(z,x)\psi_\pm(z,y)g_\pm(z)-\breve{\psi}_\pm(z,x)\psi_\pm(z,y)g_\pm(z)\\ \nn
& := G^\prime_\pm(z,x,y)+G^{\prime\prime}_\pm(z,x,y), \quad \pm y >\pm x,
\end{align*}
where $x$ and $y$ are considered as fixed parameters. As a function of $z$ it is meromorphic in the domain $\C\backslash \sigma$ with simple poles at the points $\la_k$ of the discrete spectrum. It is continuous up to the boundary $\sigma^u\cup\sigma^l$, except for the points of the set, which consists of the band edges of the background spectra $\partial\sigma_+$ and $\partial\sigma_-$, where 
\beq\label{edgeG}
G_\pm(z,x,y)=O((z-E)^{-1/2}) \quad \text{ as } \quad E\in\partial\sigma_+\cup\partial\sigma_-.
\eeq

Outside a small neighborhood of the gaps of $\sigma_+$ and $\sigma_-$, the following asymptotics as $z\to\infty$ are valid:
\begin{align}\nn
\phi_\mp(z,x)&=\E^{\mp\I\sqrt{z}x(1+O(\frac{1}{z}))}\left(1+O(z^{-1/2})\right), \quad g_\pm(z)=\frac{-1}{2\I\sqrt{z}}+O(z^{-1}),\\ \nn 
\breve{\psi}_\pm(z,x)& =\E^{\mp\I\sqrt{z}x(1+O(\frac{1}{z}))}\left(1+O(z^{-1})\right),  \quad T_\pm(z)=1+O(z^{-1/2}),\\ \nn
\psi_\pm(z,y) & = \E^{\pm\I\sqrt{z}y(1+O(\frac{1}{z}))}\left(1+O(z^{-1})\right),
\end{align}
and the leading term of $\phi_\mp(z,x)$ and $\breve \psi_\pm(z,x)$ are equal, thus  
\beq\label{estG}
G_\pm(z,x,y)=\E^{\pm\I\sqrt{z}(y-x)(1+O(\frac{1}{z}))}O(z^{-1}), \quad \pm y > \pm x.
\eeq

\begin{figure}
\begin{picture}(7,5.2)
\put(1,2.6){\line(1,0){0.5}}
\put(2.3,2.6){\line(1,0){1.3}}
\put(4.2,2.6){\line(1,0){1.8}}

\put(3.0,2.75){\vector(1,0){0.1}}
\put(3.0,2.45){\vector(-1,0){0.1}}
\put(1.2,2.75){\vector(1,0){0.1}}
\put(1.2,2.45){\vector(-1,0){0.1}}
\put(4.6,2.45){\vector(-1,0){0.1}}
\put(4.6,2.75){\vector(1,0){0.1}}	

\curve(4.4,0.2,  4.6,0.8,  4.86,2.0,   4.892,2.3, 4.894,2.45, 4.6,2.45, 4.2,2.45, 4.05,2.6, 4.2,2.75, 4.6,2.75, 4.894,2.75, 4.892,2.9,   4.86,3.2,  4.6,4.4,  4.4,5.)
\curve( 2.15,2.6, 2.3,2.75,  2.6,2.75, 3.6,2.75, 3.75,2.6, 3.6,2.45, 2.6,2.45, 2.3,2.45, 2.15,2.6)
\curve(1.,2.75, 1.2,2.75, 1.5,2.75, 1.65,2.6, 1.5,2.45, 1.2,2.45, 1.,2.45)

\end{picture}
\caption{Contours $\Gamma_{\varepsilon,n}$}
\label{figure:solreg}
\end{figure}
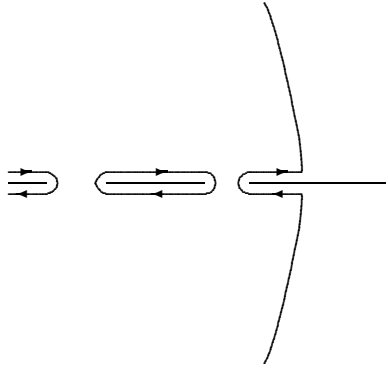

Consider the following sequence of contours $\Gamma_{\varepsilon,n,\pm}$, where $\Gamma_{\varepsilon,n,\pm}$ consists of two parts for every $n\in\N$ and $\varepsilon\geq0$:
\begin{enumerate}
\item
$C_{\varepsilon,n,\pm}$ consists of a part of a circle which is centered at the origin and has as radii the distance from the origin to the midpoint of the largest band of $[E_{2n}^\pm,E_{2n+1}^\pm]$, which lies inside $\sigma^{(2)}$, together with a part wrapping around the corresponding band of $\sigma$ at a small distance, which is at most $\varepsilon$, as indicated by figure 1.
\item
Each band of the spectrum $\sigma$, which is fully contained in $C_{\varepsilon,n,\pm}$, is surrounded by a small loop at a small distance from $\sigma$ not bigger than $\varepsilon$.
\end{enumerate}
W.l.o.g. we can assume that all the contours are non-intersecting. \\
Using the Cauchy theorem, we obtain 
\beq\nn
\frac{1}{2\pi\I}\oint_{\Gamma_{\varepsilon,n,\pm}} G_\pm(z,x,y)dz=\sum_{\la_k \in \inte(\Gamma_{\varepsilon,n,\pm})}\Res_{\la_k}G_\pm(z,x,y), \quad \varepsilon>0.
\eeq
By (\ref{edgeG}) the limit value of $G_\pm(z,x,y)$  as $\varepsilon\to 0$ is integrable on $\sigma$, and the function $G_\pm^{\prime\prime}(z,x,y)$ has no poles at the points of the discrete spectrum, thus we arrive at 
\beq\label{eqGn}
\frac{1}{2\pi\I}\oint_{\Gamma_{0,n,\pm}}G_\pm(z,x,y)dz=\sum_{\la_k \in \inte(\Gamma_{0,n,\pm})}\Res_{\la_k}G^\prime_\pm(z,x,y),\quad \pm y > \pm x.
\eeq
Estimate (\ref{estG}) allows us now to apply Jordan's lemma, when letting $n\to\infty$, and we therefore arrive, up to that point only formally, at
\beq\label{eqG}
\frac{1}{2\pi\I}\oint_{\sigma}G_\pm(\la,x,y)d\la=\sum_{\la_k\in\sigma_d}\Res_{\la_k}G_\pm^\prime(\la,x,y), \quad \pm y>\pm x.
\eeq

Next, note that the function $G_\pm^{\prime\prime}(\la,x,y)$ does not contribute to the left part of (\ref{eqG}), since $G_\pm^{\prime\prime}(\la^u,x,y)=G_\pm^{\prime\prime}(\la^l,x,y)$ for $\la\in\sigma_\mp^{(1)}$ and, hence	 $\oint_{\sigma_\mp^{(1)}}G_\pm^{\prime\prime}(\la,x,y)d\la=0$.  In addition, $\oint_{\sigma_\pm}G_\pm^{\prime\prime}(\la,x,y)d\la=0$ for $x\not =y$ by Lemma~\ref{lem1.1} (iv).

Therefore we arrive at the following equation,
\beq\label{integ2}
\frac{1}{2\pi\I}\oint_{\sigma_\pm}G_\pm^\prime(\la,x,y)d\la=\sum_{\la_k\in\sigma_d}\Res_{\la_k}G_\pm^\prime(\la,x,y), \quad \pm y>\pm x.
\eeq

To make our argument rigorous we have to show that the series of contour integrals along the parts of the spectrum contained in $C_{0,n,\pm}$ on the left hand side of \eqref{eqGn} converges as $n\to\infty$ and that the contribution of the integrals along the circles $C_{0,n,\pm}$ converges against zero as $n\to\infty$, by applying Jordan's lemma. This will be done next.   

Using (\ref{1.14}), (\ref{repphi}), (\ref{repolphi}), and (\ref{scatrel}), we obtain
\begin{align}\nn
\frac{1}{2\pi\I} & \oint_{\sigma_\pm} G_\pm^\prime(\la,x,y)d\la=
\oint_{\sigma_\pm}T_\pm(\la)\phi_\mp(\la,x)\psi_\pm(\la,y)d\rho_\pm(\la)\\ \nn 
& = \oint_{\sigma_\pm} \Big(R_\pm(\la)\phi_\pm(\la,x)+\overline{\phi_\pm(\la,x)}\Big)\psi_\pm(\la,y)d\rho_\pm(\la)\\ \nn
& =\oint_{\sigma_\pm}R_\pm(\la)\psi_\pm(\la,x)\psi_\pm(\la,y)d\rho_\pm(\la)+\oint_{\sigma_\pm} \breve{\psi}_\pm(\la,x)\psi_\pm(\la,y)d\rho_\pm(\la) \\ \nn
& \quad \pm\int_x^{\pm\infty} dt K_\pm(x,t)\Big(\oint_{\sigma_\pm} R_\pm(\la)\psi_\pm(\la,t)\psi_\pm(\la,y)d\rho_\pm(\la)+\delta(t-y)\Big)\\ \nn
& = F_{r,\pm}(x,y)\pm\int_x^{\pm\infty}K_\pm(x,t)F_{r,\pm}(t,y)dt+K_\pm(x,y), 
\end{align}
where
\begin{equation*}
F_{r,\pm}(x,y)=\oint_{\sigma_\pm} R_\pm(\la)\psi_\pm(\la,x)\psi_\pm(\la,y)d\rho_\pm(\la).
\end{equation*}

Now properties (ii) and (iii) from Lemma~\ref{lem:scat} imply that 
\begin{equation*}
\vert R_\pm(\la)\vert <1 \quad \text{for} \quad \la\in\inte(\sigma^{(2)}), \quad \vert R_\pm(\la)\vert=1 \quad \text{for} \quad \la\in\sigma_\pm^{(1)}.
\end{equation*}
and by (\ref{intrep}) we can write
\begin{align}\nn
 F_{r,\pm} (x,y) & =\oint_{\sigma_\pm}R_\pm(\la)\psi_\pm(\la,x)\psi_\pm(\la,y)d\rho_\pm(\la)\\  \nn
& = -\oint_{\sigma_\pm}R_\pm(\la)\frac{(G_\pm(\la,x)G_\pm(\la,y))^{1/2}}{2Y_\pm(\la)^{1/2}}\exp(\eta_\pm(\la,x)+\eta_\pm(\la,y))d\la,
\end{align}
with
\beq\nn
\eta_\pm(\la,x):=\pm\int_0^x \frac{Y_\pm(\la)^{1/2}}{G_\pm(\la,\tau)}d\tau \in\I\R.
\eeq

We are now ready to prove the following lemma.

\begin{lemma}\label{lem:Fr}
 The sequence of functions 
\begin{equation*}
 F_{r,n,\pm}(x,y)=\oint_{\sigma_\pm \cap \Gamma_{0,n,\pm}} R_\pm(\la)\psi_\pm(\la,x)\psi_\pm(\la,y) d\rho_\pm(\la),
\end{equation*}
is uniformly bounded with respect to $x$ and $y$, that means for all $n\in\N$, $\vert F_{r,n,\pm}(x,y)\vert \leq C$. Moreover, $F_{r,n,\pm}(x,y)$ converges uniformly as $n\to\infty$ to the function 
\begin{equation*}
 F_{r,\pm}(x,y) = \oint_{\sigma_\pm}  R_\pm(\la)\psi_\pm(\la,x)\psi_\pm(\la,y)d\rho_\pm(\la),
\end{equation*}
which is again uniformly bounded with respect to $x$ and $y$. In particular, $F_{r,\pm}(x,y)$ is continuous with respect to $x$ and $y$.
\end{lemma}

\begin{proof}
For $\la\in\sigma_\pm$ as $\la\to\infty$ we have the following asymptotic behavior 
\begin{enumerate}
\item
in a small neighborhood $V_{n}^\pm$ of $E=E_{n}^\pm$
\beq\nn
\vert R_\pm(\la)\psi_\pm(\la,x)\psi_\pm(\la,y)g_\pm(\la)\vert =O\Big(\frac{\sqrt{E_{2j}^\pm-E_{2j-1}^\pm}}{\sqrt{\la(\la-E)}}\Big),
\eeq
\item
in a small neighborhood $W_n^\pm$ of $E=E_n^\mp$, if $E\in\sigma_\pm$
\beq\nn
 R_\pm(\la)\psi_\pm(\la,x)\psi_\pm(\la,y)g_\pm(\la)=O(\frac{1}{\sqrt{\la}}),
\eeq
\item
and for $\la\in\sigma_\pm\backslash\bigcup_{n\in\N}(V_{n}^\pm\cup W_{n}^\pm)$
\beq\label{eq:asymp1}
R_\pm(\la)\psi_\pm(\la,x)\psi_\pm(\la,y)g_\pm(\la)=\exp(\pm\I\sqrt{\la}(\vert x\vert +\vert y\vert )(1+O(\frac{1}{\la})))\Big(\frac{C}{\la}+O\Big(\frac{1}{\la^{3/2}}\Big)\Big).
\eeq
\end{enumerate}
These estimates are good enough to show that $F_{r,\pm}(x,y)$ exists, if we choose $V_n^\pm$ and $W_n^\pm$ in the following way: 
We choose $V_n^\pm\subset \sigma_\pm^{(1)}\cup \sigma^{(2)}$, if $E_n^\pm$ is a band edge of $\sigma_\pm^{(1)}$, such that $V_n^\pm$ consists of the corresponding band of $\sigma_\pm^{(1)}$ together with the following part of $\sigma_\pm^{(2)}$ with length $E_n^\pm-E_{n-1}^\pm$, if $n$ is even and $E_{n+1}^\pm-E_n^\pm$, if $n$ is odd. If $E_n^+$ is a band edge of $\sigma^{(2)}$, we choose $V_n^\pm\subset\sigma^{(2)}$, where the length of $V_n^\pm$ is equal to the length of the gap pf $\sigma_\pm$ next to it. 
We set $W_n^\pm\subset\sigma^{(2)}$ with length $3(E_n^\mp-E_{n-1}^\mp)$, if $n$ is even and $3(E_{n+1}^\mp-E_n^\mp)$, if $n$ is odd, centered at the midpoint of the corresponding gap in $\sigma_\mp$. As we are working in the Levitan class and we therefore know that $\sum_{n=1}^\infty (E_{2n-1}^\pm)^{l^\pm}(E_{2n}^\pm-E_{2n-1}^\pm)<\infty$ for some $l^\pm>1$, we obtain that the sequences belonging to $V_n^\pm$ and $W_n^\pm$ converge. 

As far as the behavior along the spectrum away from the band edges of $\sigma_+$ and $\sigma_-$ is concerned observe that 
\begin{equation*}
\vert \exp(\pm\I\sqrt{\la}(x+y)O(\frac{1}{\la}))\vert\leq 1+ (x+y)O(\frac{1}{\sqrt{\lambda}}),\quad \lambda\to\infty.
\end{equation*}
and therefore 
\begin{equation}
 R_\pm(\la)\psi_\pm(\la,x)\psi_\pm(\la,y)g_\pm(\la)=\exp(\pm\I\sqrt{\la}(x+y))\Big(\frac{C}{\la}+(1+x+y)O\Big(\frac{1}{\la^{3/2}}\Big)\Big).
\end{equation}

To show the convergence of the series $F_{r,\pm}(x,y)$ for fixed $x$ and $y$, we split the integral along the spectrum $\sigma$ up into three integrals along $\bigcup_{n\in\N} V_n^\pm$, $\bigcup_{n\in\N} W_n^\pm$, and $\sigma\backslash\bigcup_{n\in\N}(V_n^\pm\cup W_n^\pm)$ respectively.

As far as the integral along $\bigcup_{n\in\N} V_n^pm$ is concerned observe that the integrand has a square root singularity at the boundary and is therefore integrable along $V_n^\pm$ for all $n\in\N$. Since we are working within the Levitan class the sum over all $n\in\N$ converges.

The integrand can be uniformly bounded for all $\lambda \in W_n^\pm$ such that $\lambda\geq 1$. Since there are only finitely many $n\in\N$ such that $W_n^\pm\subset [0,1]$, the corresponding series converges by the definition of the Levitan class.  

Thus it is left to consider the integral along $\sigma\backslash\bigcup_{n\in\N}(V_n^\pm\cup W_n^\pm)$:=I.
Direct computation yields 
\begin{align}\nn
 \int_a^b \exp(\pm\I\sqrt{\la}(x+y))\frac{C}{\la}d\la& =\pm\exp(\pm\I\sqrt{\la}(x+y))\frac{2C}{\I\la^{1/2}(x+y)}\vert_a^b\\ \nn 
& \pm\int_a^b\exp(\pm\I\sqrt{\la}(x+y))\frac{C}{\I\la^{3/2}(x+y)}d\la,
\end{align}
which is finite since by assumption $\pm x\leq\pm y$.
Hence one possibility to see that the corresponding series of integrals converges is to integrate first the function describing the asymptotic behavior along $[E_0^\pm,\infty]$ and substract from it the series of integrals corresponding to the $[E_0^\pm,\infty]\cap I^c$. Since every interval belonging to the complement belongs to a neighborhood of the gaps of $\sigma^{(2)}$ and the integrand can be uniformly bounded, the definition of the Levitan class implies that this series converges. 

Similarly we conclude 
\begin{align}\nn
 \int_a^b \exp(\pm\I\sqrt{\la}(x+y))(x+y)O(\frac{1}{\la^{3/2}})d\la& =\exp(\pm\I\sqrt{\la}(x+y))O(\frac{1}{\la})\vert_a^b\\ \nn
& +\int_a^b \exp(\pm\I\sqrt{\la}(x+y))O(\frac{1}{\la^{2}})d\la
\end{align}

Note that since we are working within the Levitan class all estimates are independent of $x$ and $y$.
\end{proof}

For investigating the other terms, we will need the following lemma, which is taken from \cite{F1}:
\begin{lemma}\label{lem:F1}
 Suppose in an integral equation of the form
\beq\label{inteq}
f_\pm(x,y)\pm\int_x^{\pm\infty}K_\pm(x,t)f_\pm(t,y)dt=g_\pm(x,y), \quad \pm y>\pm x,
\eeq 
the kernel $K_\pm(x,y)$ and the function $g_\pm(x,y)$ are continuous for $\pm y>\pm x$, 
\beq\nn
\vert K_\pm(x,y) \vert \leq C_\pm(x) Q_\pm(x+y), 
\eeq 
and for $g_\pm(x,y)$ one of the following estimates hold
\beq\label{estg1}
\vert g_\pm(x,y)\vert \leq C_\pm(x)Q_\pm(x+y), \quad \text{ or }
\eeq
\beq\label{estg2}
\vert g_\pm(x,y) \vert \leq C_\pm(x)(1+\max (0,\pm x)).
\eeq
Furthermore assume that 
\begin{equation*}
\pm\int_0^{\pm\infty} (1+\vert x \vert^2)\vert q(x)-p_\pm(x)\vert dx<\infty. 
\end{equation*}
Then \eqref{inteq} is uniquely solvable for $f_\pm(x,y)$. The solution $f_\pm(x,y)$ is also continuous in the half-plane 
$\pm y >\pm x$, and for it the estimate \eqref{estg1} respectively \eqref{estg2} is reproduced.

Moreover, if a sequence $g_{n,\pm}(x,y)$ satisfies \eqref{estg1} or \eqref{estg2} uniformly with respect to $n$ and 
pointwise $g_{n,\pm}(x,y)\to 0$, for $\pm y>\pm x$, then the same is true for the corresponding sequence of solutions 
$f_{n,\pm}(x,y)$ of \eqref{inteq}.
\end{lemma}

\begin{proof}
 For a proof we refer to \cite[Lemma 6.3]{F1}.
\end{proof}

\begin{remark}
 An immediate consequence of this lemma is the following. If $\vert g_\pm(x,y)\vert \leq C_\pm(x)$, where $C_\pm(x)$ denotes a bounded function, then $\vert g_\pm(x,y)\vert\leq C_\pm(x)(1+\max (0,\pm x))$ and therefore $\vert f_\pm(x,y)\vert \leq C_\pm(x)(1+\max (0,\pm x))$. Rewriting this integral equation as follows
\begin{equation*}
 f_\pm(x,y)=g_\pm(x,y)\mp\int_x^{\pm\infty}K_\pm(x,t)f_\pm(t,y)dy,
\end{equation*}
we obtain that the absolute value of the right hand side is smaller than a bounded function $\tilde C_\pm(x)$ by using \eqref{moment} and \eqref{estK}, and hence the same is true for the left hand side. In particular if $C_\pm(x)$ is a decreasing function the same will be true for $\tilde C_\pm(x)$.
\end{remark}

We will now continue the investigation of our integral equation.

\begin{lemma}\label{lem:Fh}
 The sequence of functions 
\begin{equation*}
 F_{h,n,\pm}(x,y)=\int_{\sigma_\pm^{(1),u}\cap\Gamma_{0,n,\pm}} \vert T_\mp(\la)\vert ^2 \psi_\pm(\la,x)\psi_\pm(\la,y)d\rho_\mp(\la)
\end{equation*}
is uniformly bounded, that means for all $n\in\N$, $\vert F_{h,n,\pm}(x,y)\vert \leq C_\pm(x)$,
where $C_\pm(x)$ are monotonically decrasing functions as $x\to\pm\infty$. Moerover, $F_{h,n,\pm}(x,y)$ converges uniformly as $n\to\infty$ to the function
\begin{equation*}
F_{h,\pm}(x,y) =\int_{\sigma_\mp^{(1),u}}   \vert T_\mp(\la) \vert ^2 \psi_\pm(\la,x)\psi_\pm(\la,y) d\rho_\mp(\la), 
\end{equation*}
which is again bounded by some monotonically increasing function. In particular, $F_{h,\pm}(x,y)$ is continuous with respect to $x$ and $y$.
\end{lemma}

\begin{proof}
On the set $\sigma_\mp^{(1)}$ both the numerator and the denominator of the function $G_\pm^\prime(\la,x,y)$ have poles (resp. square root singularities) at the points of the set $\sigma_\mp^{(1)}\cap(M_\pm\cup(\partial\sigma_+^{(1)}\cap\partial\sigma_-^{(1)}))$ (resp. $\sigma_\mp^{(1)}\cap(M_\mp\backslash(M_\mp\cap M_\pm))$ , but multiplying them, if necessary away, we can avoid singularities.
Hence, w.l.o.g., we can suppose $\sigma_\mp^{(1)}\cap(M_{r,+}\cup M_{r,-})=\emptyset$. 
Thus we can write 
\begin{equation*}
\frac{1}{2\pi\I}\oint_{\sigma_\mp^{(1)}}G_\pm^\prime(\la,x,y)d\la
=\frac{1}{2\pi\I}\oint_{\sigma_\mp^{(1)}}T_\pm(\la)\phi_\mp(\la,x)\psi_\pm(\la,y)g_\pm(\la)d\la.
\end{equation*}

For investigating this integral we will consider, using (\ref{scatrel}),
\begin{align*}
 \frac{1}{2\pi\I}\oint_{\sigma_\mp^{(1)}}T_\mp(\la) &\phi_\mp(\la,x)\phi_\pm(\la,y)g_\mp(\la)d\la \\ \nn
& = \frac{1}{2\pi\I}\oint_{\sigma_\mp^{(1)}}\phi_\mp(\la,x)\Big(\overline{\phi_\mp(\la,y)}
+R_\mp(\la)\phi_\mp(\la,y)\Big)g_\mp(\la)d\la.
\end{align*}
First of all note that the integrand, because of the representation on the right hand side, can only have square root singularities at the boundary $\partial\sigma_\mp^{(1)}$ and we therefore have  
\begin{align}\nn
\int_{\sigma_\mp^{(1)}\cap [E_{2n-1}^\pm,E_{2n}^\pm]}& \vert \phi_\mp(\la,x) \Big(\overline{\phi_\mp(\la,y)}+ R_\mp(\la)\phi_\mp(\la,y)\Big)g_\mp(\la)\vert d\la\\ \nn
 & \leq 2\int_{\sigma_\mp^{(1)}\cap [E_{2n-1}^\pm, E_{2n}^\pm]}\vert \phi_\mp(\la,x)\phi_\mp(\la,y)g_\mp(\la)\vert d\la \\ \nn 
& \leq C_\pm(y)C_\pm(x)\left( \frac{(E_{2n}^\pm-E_{2n-1}^\pm)}{\sqrt{\la-E_0^\mp}}+\frac{\sqrt{E_{2n}^\pm-E_{2n-1}^\pm}}{\sqrt{\la-E_0^\mp}}\right),
\end{align}
where $E_{2n-1}^\pm$ and $E_{2n}^\pm$ denote the edges of the gap of $\sigma_\pm$ in which the corresponding part of $\sigma_\mp^{(1)}$ lies and $C_\pm(x)$ denote monotonically decreasing functions from now on. Therefore as we are working in the Levitan class and by separating $\sigma_\mp^{(1)}$ into the different parts, one obtains that 
\begin{equation*}
\vert\frac{1}{2\pi\I}\oint_{\sigma_\mp^{(1)}}T_\pm(\la)\phi_\mp(\la,x)\phi_\pm(\la,y)g_\pm(\la)d\la\vert\leq C_\pm(y)C_\pm(x).
\end{equation*}
Thus we can now apply Lemma~\ref{lem:F1}, and hence
\begin{equation*}
\vert \frac{1}{2\pi\I}\oint_{\sigma_\mp^{(1)}}T_\pm(\la)\phi_\mp(\la,x)\psi_\pm(\la,y)g_\pm(\la)d\la\vert\leq C_\pm(y)C_\pm(x)(1+\max(0,\pm y)).
\end{equation*}
Note that we especially have, because of (\ref{estK}),   
\begin{equation*}
\vert\frac{1}{2\pi\I}\oint_{\sigma_\mp^{(1)}}T_\pm(\la)\phi_\mp(\la,x)\phi_\pm(\la,y)g_\pm(\la)d\la\vert\leq  C_\pm(x)
\end{equation*}
Therefore we can conclude that for fixed $x$ and $y$ the left hands side of (\ref{integ2}) exists and satisfies 
\begin{equation*}
\vert \frac{1}{2\pi\I}\oint_{\sigma_\mp^{(1)}}T_\pm(\la)\phi_\mp(\la,x)\psi_\pm(\la,y)g_\pm(\la)d\la\vert\leq C_\pm(x),
\end{equation*}
and hence 
\begin{equation*}
\vert \frac{1}{2\pi\I}\oint_{\sigma_\mp^{(1)}}G'_\pm(z,x,y)dz\vert\leq C_\pm(x). 
\end{equation*}

We will now rewrite the integrand in a form suitable for our further purposes. 
Namely, since $\psi_\pm(\la,x)\in\R$ as $\la\in\sigma_\mp^{(1)}$, we have 
\begin{equation*}
\frac{1}{2\pi\I}\oint_{\sigma_\mp^{(1)}}G_\pm^\prime(\la,x,y)d\la=\frac{1}{2\pi\I}\int_{\sigma_\mp^{(1),u}}\psi_\pm(\la,y)
\Big(\frac{\overline{\phi_\mp(\la,x)}}{\overline{W(\la)}}-\frac{\phi_\mp(\la,x)}{W(\la)}\Big)d\la
\end{equation*}
Moreover, (\ref{scatrel}) and Lemma~\ref{lem:scat} (ii) imply 
\beq\nn
\overline{\phi_\mp(\la,x)}=T_\mp(\la)\phi_\pm(\la,x)-\frac{T_\mp(\la)}{\overline{T_\mp(\la)}}.
\eeq
Therefore,
\begin{align}\label{phiW}
\frac{\phi_\mp(\la,x)}{W(\la)}-\frac{\overline{\phi_\mp(\la,x)}}{\overline{W(\la)}}& =\phi_\mp(\la,x)\Big(\frac{1}{W(\la)}
+\frac{T_\mp(\la)}{\overline{T_\mp(\la)W(\la)}}\Big)-\frac{T_\mp(\la)\phi_\pm(\la,x)}{\overline{W(\la)}}\\ \nn
& =\phi_\mp(\la,x)\frac{2\Re (T_\mp^{-1}(\la)\overline{W(\la)})T_\mp(\la)}{\vert W(\la)\vert^2}-\frac{T_\mp(\la)\phi_\pm(\la)}{\overline{W(\la)}}.
\end{align}
But by (\ref{eq:TgW}) 
\begin{equation*}
T_\mp^{-1}(\la)\overline{W(\la)}=\vert W(\la)\vert^2g_\mp(\la)\in\I\R, \quad \text{ for }\la\in\sigma_\mp^{(1)},
\end{equation*}
and therefore the first summand of  (\ref{phiW}) vanishes. Using now $\overline{W}=(\overline{T_\mp}g_\mp)^{-1}$ we arrive at 
\begin{equation*}
\frac{\overline{\phi_\mp(\la,x)}}{\overline{W(\la)}}-\frac{\phi_\mp(\la,x)}{W(\la)}=\vert T_\mp(\la)\vert^2g_\mp(\la)\phi_\pm(\la,x)
\end{equation*}
and hence
\begin{equation*}
\frac{1}{2\pi\I}\oint_{\sigma_\mp^{(1)}}G_\pm^\prime(\la,x,y)d\la=F_{h,\pm}(x,y)\pm\int_x^{\pm\infty}K_\pm(x,t)F_{h,\pm}(t.y)dt,
\end{equation*}
where
\begin{equation*}
F_{h,\pm}(x,y)=\int_{\sigma_\pm^{(1),u}}\vert T_\mp(\la)\vert^2\psi_\pm(\la,x)\psi_\pm(\la,y)d\rho_\mp(\la),
\end{equation*}
and 
\begin{equation*}
\vert F_{h,\pm}(x,y)\vert \leq C_\pm(x)C_\pm(y)
\end{equation*}
by Lemma~\ref{lem:F1}. The partial sums $F_{h,n,\pm}(x,y)$ can be investigated similarly
\end{proof}

We will now investigate the r.h.s. of (\ref{eqGn}) and \eqref{integ2}. Therefore we consider first the question of the existence of the right hand side: 

To prove the boundedness of the corresponding series on the left hand side, it is left to investigate the series, which correspond to the circles. We will derive the necessary estimates only for the part of the n'th circle $K_{R_{n,\pm}}$, where $R_{n,\pm}$ denotes the radius, in the upper half plane as the part in the lower half plane can be considered similarly. We have 
\begin{align}\nn
 \vert \int_{K_{R_{n},\pm}}G_\pm(z,x,y)dz\vert  
& \leq \int_0^\pi C \E^{\pm\sqrt{R}(x-y)(1-\nu)\sin(\theta/2)}d\theta \\ \nn
& \leq \int_0^{\pi/2} C \E^{\pm\sqrt{R}(x-y)(1-\nu)\sin(\eta)} d\eta\\ \nn 
& \leq \int_0^{\pi/2} C \E^{\pm\sqrt{R}(x-y)(1-\nu) 2\frac{\eta}{\pi}}d\eta\\ \nn
&\leq C\frac{1}{\sqrt{R}(x-y)(1-\nu)}\E^{\pm\sqrt{R}(x-y)(1-\nu)2\frac{\eta}{\pi}}\vert_0^{\frac{\pi}{2}}, 
\end{align}
where $C$ and $\nu$ denote some constant, which are dependent on the radius (cf. Lemma~\ref{lempsi}).
Therefore as already mentioned the part belonging to the circles converges against zero and hence the same is true for the corresponding series, by Jordan's lemma.
 
Thus we obtain that the sequence of partial sums on the right hand side of \eqref{eqGn} and \eqref{integ2} is uniformly bounded and we are therefore ready to prove the following result.

\begin{lemma}\label{lem:Fd}
 The sequence of functions
\begin{equation*}
 F_{d,n,\pm}(x,y)=\sum_{\la_k\in\sigma_d\cap \Gamma_{0,n,\pm}}(\gamma_k^\pm)^2\tilde \psi_\pm(\la_k,x)\tilde \psi_\pm(\la_k,y)
\end{equation*}
is uniformly bounded, that means for all $n\in\N$, 
$\vert F_{d,n,\pm}(x,y)\vert \leq C_\pm(x)$,
where $C_\pm(x)$ are monotonically increasing functions. Moreover, $F_{d,n,\pm}(x,y)$ converges uniformly as $n\to\infty$ to the function
\begin{equation*}
 F_{d,\pm}(x,y) = \sum_{\la_k\in\sigma_d}(\gamma_k^\pm)^2\tilde \psi_\pm(\la_k,x)\tilde \psi_\pm(\la_k,y), 
\end{equation*}
which is again bounded by some monotonically increasing function. In particular, $F_{d,\pm}(x,y)$ is continuous with respect to $x$ and $y$.  
\end{lemma}

\begin{proof}
Applying (\ref{repphi}), (\ref{deftihW}), (\ref{deftiphi}), (\ref{lidep}), and (\ref{cga})to the right hand side of \eqref{integ2}, yields
\begin{align}\nn
\sum_{\la_k\in\sigma_d} \Res_{\la_k}G_\pm^\prime(\la,x,y)
&=-\sum_{\la_k\in\sigma_d}\Res_{\la_k}\frac{\tilde{\phi}_\mp(\la,x)\tilde{\psi}_\pm(\la,y)}{\tilde{W}(\la)} \\ \nn
& = -\sum_{\la_k\in\sigma_d}\frac{\tilde{\phi}_\pm(\la_k,x)\tilde{\psi}_\pm(\la_k,y)}{\tilde{W}^\prime(\la_k)c_{k,\pm}}\\ \nn
& =-\sum_{\la_k\in\sigma_d}(\gamma_k^\pm)^2\tilde{\phi}_\pm(\la_k,x)\tilde{\psi}_\pm(\la_k,y) \\ \nn
&= -F_{d,\pm}(x,y)\mp\int_x^{\pm\infty}K_\pm(x,t)F_{d,\pm}(t,y)dt,
\end{align}
where 
\begin{equation*}
F_{d,\pm}(x,y):=\sum_{\la_k\in\sigma_d} (\gamma_k^\pm)^2\tilde{\psi}_\pm(\la_k,x)\tilde{\psi}_\pm(\la_k,y).
\end{equation*}

Thus we obtained the following integral equation,
\begin{align}\label{eqFd}
F_{d,\pm}(x,y)& = -K_\pm(x,y)-F_{c,\pm}(x,y)\mp\int_x^{\pm\infty}K_\pm(x,t)F_{c,\pm}(t,y)dt\\ \nn
& \mp\int_x^{\pm\infty}K_\pm(x,t)F_{d,\pm}(t,y)dt,
\end{align}
which we can now solve for $F_{d,\pm}(x,y)$ using again Lemma~\ref{lem:F1} and hence $F_{d,\pm}(x,y)$ exists and satisfies the given estimates.
The corresponding partial sums can be investigated analogously using the considerations from above. 
\end{proof}

Putting everything together, we see that we have obtained the GLM equation.
\begin{theorem}
The GLM equation has the form 
\beq\label{GLM}
K_\pm(x,y)+F_\pm(x,y)\pm\int_x^{\pm\infty}K_\pm(x,t)F_\pm(t,y)dt=0, \quad \pm(y-x)>0,
\eeq
where
\begin{align}\label{defFpm}
F_\pm(x,y)& = \oint_{\sigma_\pm}R_\pm(\la)\psi_\pm(\la,x)\psi_\pm(\la,y)d\rho_\pm(\la) \\ \nn
& +\int_{\sigma_\mp^{(1),u}}  \vert T_\mp(\la)\vert^2 \psi_\pm(\la,x)\psi_\pm(\la,y)d\rho_\mp(\la) \\ \nn
& +\sum_{k=1}^\infty (\gamma^\pm_k)^2 \tilde\psi_\pm(\la_k,x)\tilde\psi_\pm(\la_k,y).
\end{align}
\end{theorem}

Moreover, we have
\begin{lemma}\label{lem:pF}
The function $F_\pm(x,y)$ is continuously differentiable with respect to both variables and there exists a real-valued function $q_\pm(x)$, $x\in\R$ with 
\begin{equation*}
\pm\int_a^{\pm\infty}(1+x^2)\vert q_\pm(x)\vert dx<\infty, \quad \text{ for all } a\in\R,
\end{equation*}
such that 
\beq\label{estF}
\vert F_\pm(x,y)\vert \leq \tilde C_\pm(x)Q_\pm(x+y),
\eeq
\beq\label{estpF}
\left\vert \frac{d}{dx}F_\pm(x,y)\right\vert \leq \tilde C_\pm(x)\left(\left\vert q_\pm\left(\frac{x+y}{2}\right)\right\vert +Q_\pm(x+y)\right),
\eeq
\beq\label{Fxx}
\pm\int_a^{\pm\infty}\left\vert \frac{d}{dx}F_\pm(x,x)\right\vert (1+x^2)dx<\infty,
\eeq
where 
\begin{equation*}
Q_\pm(x)=\pm\int_{\frac{x}{2}}^{\pm\infty}\vert q_\pm(t) \vert dt,
\end{equation*}
and $\tilde C_\pm(x)>0$ is a continuous function, which decreases monotonically as $x\to\pm\infty$.
\end{lemma}

\begin{proof}
Applying once more Lemma~\ref{lem:F1}, one obtains (\ref{estF}). Now, for simplicity, we will restrict our considerations to the + case and omit + whenever possible. Set $Q_1(u)=\int_u^\infty Q(t)dt$. Then, using (\ref{decay}), the functions $Q(x)$ and $Q_1(x)$ satisfy 
\begin{equation*}
\int_a^\infty Q_1(t)dt<\infty, \quad \int_a^\infty Q(t)(1+\vert t\vert )dt<\infty. 
\end{equation*}
Differentiating (\ref{GLM}) with respect to $x$ and $y$ yields 
\begin{equation*}
\vert F_x(x,y)\vert \leq \vert K_x(x,y)\vert+\vert K(x,x)F(x,y)\vert +\int_x^\infty\vert K_x(x,t)F(t,y)\vert dt,
\end{equation*}
\begin{equation*}
F_y(x,y)+K_y(x,y)+\int_x^\infty K(x,t)F_y(t,y)dt=0.
\end{equation*}
We already know that the functions $Q(x)$, $Q_1(x)$, $C(x)$, and $\tilde C(x)$ are monotonically decreasing and positive. Moreover,
\begin{equation*}
\int_x^\infty \Big(\left\vert q_+\Big(\frac{x+t}{2}\Big)\right\vert +Q(x+t)\Big)Q(t+y)dt\leq (Q(2x)+Q_1(2x))Q(x+y),
\end{equation*}
thus we can estimate $F_x(x,y)$ and $F_y(x,y)$ can be estimates using (\ref{estpK}) and the method of successive approximation. It is left to prove (\ref{Fxx}). Therefore consider (\ref{GLM}) for $x=y$ and differentiate it with respect to $x$:
\begin{equation*}
\frac{d F(x,x)}{dx}+\frac{d K(x,x)}{dx}-K(x,x)F(x,x)+\int_x^\infty (K_x(x,tF(t,x)+K(x,t)F_y(t,x))dt=0.
\end{equation*}
Next (\ref{estK}) and (\ref{estF}) imply 
\begin{equation*}
\vert K(x,y)F(x,x)\vert \leq \tilde C(a)C(a)Q^2(2x), \quad \text{ for } x>a,
\end{equation*}
where $\int_a^\infty (1+x^2)Q^2(2x)dx<\infty$. Moreover, by (\ref{estpK}) and (\ref{estpF})
\beq\nn
\left|K_x^\prime (x,t) F(t,x)\right| + \left|K(x,t) F_y^\prime
(t,x)\right|\leq 4\tilde C(a)
\hat C(a)\Bigl\lbrace\Bigl\lvert q\Bigl(\frac{x+t}{2}\Bigr)\Bigr\rvert Q(x+t) +
Q^2(x+t)\Bigr\rbrace,
\eeq
together with the estimates
\begin{align*}
& \int_a^\infty d x\,x^2\int_x^\infty Q^2(x+t)d t\leq\int_a^\infty
|x| Q(2x)d x\ \sup_{x\geq a}\int_x^\infty|x+t| Q(x+t)d t<\infty,\\
& \int_a^\infty x^2\int_x^\infty\Bigl\lvert q\Bigl(\frac{x+t}{2}\Bigr)\Bigr\rvert
Q(x+t)d t\leq\\
& \qquad
\leq\int_a^\infty Q(2x)d x \ \sup_{x\geq a} \int_x^\infty
\Bigl\lvert q\Bigl(\frac{x+t}{2}\Bigr)\Bigr\rvert(1 +(x+t)^2)d t<\infty,
\end{align*}
and (\ref{Kxx}), we arrive at (\ref{Fxx}).

\end{proof}

In summary, we have obtained the following necessary conditions for the scattering data:
\begin{theorem}
The scattering data 
\begin{align}\nn
{\mathcal S} = \Big\{ & R_+(\lambda),\,T_+(\lambda),\, \lambda\in\sigma_+^{\mathrm{u,l}}; \,
R_-(\lambda),\,T_-(\lambda),\, \lambda\in\sigma_-^{\mathrm{u,l}};\\\label{4.6}
& \lambda_1,,\lambda_2,\dots \in\mathbb{R}\setminus (\sigma_+\cup\sigma_-),\,
\gamma_1^\pm, \gamma_2^\pm,\dots \in\mathbb{R}_+\Big\}
\end{align} 
possess the properties listed in Theorem~\ref{lem:scat}, \ref{thm:asym}, \ref{thm:TW}, and \ref{thm:pR}, and Lemma~\ref{lem:Fr}, \ref{lem:Fh}, and \ref{lem:Fd}. The functions $F_\pm(x,y)$ defined in \eqref{defFpm}, possess the properties listed in Lemma~\ref{lem:pF}.
\end{theorem}

\subsection*{Acknowledgements}

I want to thank Ira Egorova and Gerald Teschl for many discussions on this topic.

\end{document}